\documentclass[a4paper,10pt]{article}

\usepackage{geometry}
\usepackage{amsmath}
\usepackage{amssymb}
\usepackage{amsthm}
\usepackage[latin1]{inputenc}
\usepackage{eurosym}
\usepackage[dvips]{graphics}
\usepackage{graphicx}
\usepackage{epsfig}
\usepackage{enumitem}
\usepackage{mdwlist}

\usepackage{ifthen}

\renewcommand{\div}{\operatorname{div}}

\newcommand{\osc}{\operatorname{\bf{osc}}}

\newcommand{\Rr}{{\mathbb{R}}}

\newcommand{\Nn}{{\mathbb{N}}}
\newcommand{\Zz}{{\mathbb{Z}}}
\newcommand{\Tt}{{\mathbb{T}}}

\newcommand{\Hh}{{\overline{H}}}

\newcommand{\Pp}{{\mathcal{P}}}

\newcommand{\finedim}{{\unskip\nobreak\hfil\penalty50
   \hskip2em\hbox{}\nobreak\hfil\mbox{$\Box$ \qquad}
   \parfillskip=0pt \finalhyphendemerits=0\par\medskip}}

\theoremstyle{plain}
\numberwithin{equation}{section}

\newtheorem{clm}{Step}

\newtheorem{teo}{Theorem}[section]
\newtheorem{cor}[teo]{Corollary}

\newtheorem{pro}[teo]{Proposition}

\newtheorem{rem}{Remark}

\newcommand{\dm}{\emph{d} m}
\newcommand{\ep}{\epsilon}
\newcommand{\dx}{\emph{d} x}
\newcommand{\dt}{\emph{d} t}
\newcommand{\dy}{\emph{d} y}

\begin{document}

\title{On the existence of classical solutions  for stationary extended mean field games}
\author{Diogo A. Gomes\footnote{Center for Mathematical Analysis, Geometry, and Dynamical Systems, Departamento de Matem\'atica, Instituto Superior T\'ecnico, 1049-001 Lisboa, Portugal
and
King Abdullah University of Science and Technology (KAUST), CSMSE Division , Thuwal 23955-6900. Saudi Arabia. e-mail: dgomes@math.ist.utl.pt}, Stefania Patrizi\footnote{Center for Mathematical Analysis, Geometry, and Dynamical Systems, Departamento de Matem\'atica, Instituto Superior T\'ecnico, 1049-001 Lisboa, Portugal. e-mail: stefaniapatrizi@yahoo.it}, Vardan
Voskanyan\footnote{Center for Mathematical Analysis, Geometry, and Dynamical Systems, Departamento de Matem\'atica, Instituto Superior T\'ecnico, 1049-001 Lisboa, Portugal. e-mail: vartanvos@gmail.com}}

\date{\today} %%

\maketitle

\begin{abstract}
In this paper we consider extended stationary mean field games, that is mean-field games
which depend
on the velocity field of the players.
We prove various a-priori estimates which generalize
the results for quasi-variational mean field games in \cite{GMP}. In addition
we use adjoint method techniques to obtain higher regularity bounds. Then we establish existence of smooth solutions under fairly general conditions by applying the continuity method.
When applied to standard stationary mean-field games as in \cite{ll1}, \cite{GM} or \cite{GMP}
this paper yields various new estimates and regularity properties not available
previously. We discuss additionally several examples where existence of
classical solutions can be proved.
\end{abstract}

\section{Introduction}

In an attempt to understand the limiting behavior of systems involving very large
numbers of rational agents behaving non-cooperatively and under symmetry assumptions,
Lasry and Lions
\cite{ll1, ll2, ll3, ll4}, and, independently, and around the same time
Huang, Malham{\'e}, and Caines \cite{Caines1},
\cite{Caines2},
introduced a class of models called
mean field games. These problems attracted the attention of many other researchers
 and the progress has been quite fast, for recent surveys see \cite{llg2}, and \cite{cardaliaguet} and references therein.
%
% Mean field games
%arise in various applications, including growth theory in economics \cite{llg1}, environmental policy \cite{lst},
%finance, dynamics of crowds as well as in biology and social
%sciences.
%There
%is also a growing interest in numerical methods for
%these problems \cite{lst}, \cite{DY}, \cite{CDY}. The author and his collaborators \cite{GMS} have also considered the discrete time, finite state
%problem, the continuous time finite state problem \cite{GMS2}.
%Several additional problems have been worked out in detail in \cite{GueantT}, \cite{Ge}.

Denote by $\Tt^d$ the $d$-dimensional torus, and $\mathcal{P}(\mathbb{T}^d)$  the set of Borel probability measures on $\mathbb{T}^d$ and let $\mathcal{P}^{ac}(\mathbb{T}^d)$ the set of measures from $\mathcal{P}(\mathbb{T}^d)$ which are absolutely continuous. Let
\[
F\colon\mathbb{T}^d\times\mathbb{R}^d\times \mathcal{P}^{ac}(\mathbb{T}^d)\to \mathbb{R}
\]
be a function satisfying appropriate continuity, differentiability and growth conditions.
%, as will be discussed in the Section \ref{asec}.
An important class of stationary mean field games, see for instance \cite{ll1}, can be modeled
by a system of PDE's of the form
\begin{equation}
\label{mfg0}
\begin{cases}
\Delta v(x) +F(x,Dv(x),f)=\overline{F}\\
\Delta v(x)- \div(D_pF(x,Dv(x),f)f(x))=0.
\end{cases}
\end{equation}
To avoid additional difficulties it is usual to consider periodic boundary data, or equivalently, taking $x\in \Tt^d$.
%, however analogous results may hold on bounded domains with suitable boundary conditions.
The unknowns of the previous PDE are a triplet $(\overline{F}, v, f)$ where $\overline{F}$ is a real number, $v\in C^2(\Tt^d)$,
and $f\in \Pp(\Tt^d)$.

Stationary mean-field games have an independent interest but also, as shown in \cite{CLLP} (see also \cite{GMS}, and \cite{GMS2},
for discrete state problems)
they encode the asymptotic long time behavior of various mean-field games. Equations of the form \eqref{mfg0}
also arise in calculus of variations problems. One important example is the following:
given $H_0\colon\mathbb{T}^d\times\mathbb{R}^d\to\Rr$ consider the stochastic Evans-Aronsson problem
\[
\inf\limits_{\phi}\int_{\mathbb{T}^d}e^{\ep\Delta\phi+H_0(x,D\phi(x)) }\dx,
\]
where the minimization is taken over all $\phi\in C^2(\Tt^d)$.
The Euler-Lagrange for this functional can be written as
\[
\begin{cases}
\ep\Delta u(x) +H_0(x,Du(x))=\ln m(x)+\overline{H}\\
\ep\Delta m(x)- \div(D_pH_0(x,Du(x))m(x))=0,
\end{cases}
\]
where $\Hh$ is an additional parameter chosen so that $m=e^{\Delta u(x) +H_0(x,Du(x))-\Hh}$ is a probability measure in $\Tt^d$. When $\ep=0$ this problem was studied in \cite{E1} (see also \cite{GISMY}) and the case $\ep>0$
in \cite{GM} (for $d\leq 3$ or quadratic Hamiltonians in arbitrary dimension). A natural generalization
of these problems is
the so called class of quasi-variational mean-field games,
considered in \cite{GMP},
which consists in mean-field games which are perturbations of mean field games with a variational structure.

In this paper we consider a further extension of the mean field problem \eqref{mfg0} which allows the cost function of a player to depend also on
the velocity field of the players. In order to do so,
denote by $\chi(\mathbb{T}^d)$ the set of continuous vector fields on $\mathbb{T}^d.$
Let
\[
H\colon\mathbb{T}^d\times\mathbb{R}^d\times \mathcal{P}^{ac}(\mathbb{T}^d)\times \chi(\mathbb{T}^d)\to \mathbb{R}
\]
be a function satisfying appropriate conditions as we detail in Section \ref{asec}.
We consider the following equation on the $d-$dimensional torus $\mathbb{T}^d.$
\begin{equation}
\label{main}
\begin{cases}
\Delta u(x) +H(x,Du(x),m,V)=\overline{H}\\
\Delta m(x)- \div(V(x)m(x))=0\\
V(x)=D_pH(x,Du(x),m,V).
\end{cases}
\end{equation}
The unknowns for this problems are $u\colon \mathbb{T}^d\to \Rr$, identified with a $\Zz^d$-periodic function on $\Rr^d$ whenever convenient, a probability measure $m\in\mathcal{P}(\mathbb{T}^d)$, the effective Hamiltonian $\overline{H}\in \Rr$ and the effective velocity field $V\in\chi (\mathbb{T}^d).$
We require $m$ to be a probability measure absolutely continuous with respect to Lebesgue measure with strictly positive density.
%We assume that $H$ is $C^2$ in the variables $x$ and $p$ and discussed in detail in Section \ref{asec}.

An example problem is the following:
\begin{equation}
\label{mainexmp}
H(x,p,m,V)=\frac{1}{2}|p|^2-\alpha p\int_{\mathbb{T}^d} V(y)\dm(y)-g(m),
\end{equation}
 with
$\alpha$ small enough, where $g:\Rr_0^+\to \Rr$ is an increasing function,
typically $g(m)=\ln m$ or $g(m)=m^\gamma$.
% (see Section \ref{asec} for additional hypothesis).
%Here by choosing $\alpha$ small we get small $\delta,$ in the previous Assumptions.

The main result in this paper is Theorem \ref{maint} which establishes the existence of classical
solutions for \eqref{main} for a general class of Hamiltonians $H$ of which \eqref{mainexmp}
is a main example. In particular, in the case $g(m)=\ln m$ we obtain smooth solutions in any dimension $d$; for the case $g(m)=m^\gamma$, $\gamma> 0$, our results yield smooth solutions
for $d\leq 4$, and, in general, for
$\gamma\leq \frac 1 {d-4}$ if $d\geq 5$.

To the best of our knowledge, all previous results in the literature
for mean-field game do not consider the dependence on $V$. However,
even without this dependence this paper extends substantially previous
results.
In \cite{ll3} Lions and Lasry considered mean field games with Lipschitz (with respect to Wasserstein metric) nonlinearities (see also \cite{ll2} and the notes P.Cardaliaguet \cite{ cardaliaguet} for a detailed proof);
additionally, in the same paper,
the existence of solutions in Sobolev spaces for time dependent problems was also considered. In the stationary setting related estimates are discussed in the present paper
in section \ref{elmes} as a preliminary step towards additional regularity.
In \cite{GM} and \cite{GMP} the variational and quasivariational settings for stationary mean field games were considered.
In \cite{GM} the $g(m)=\ln m$ was addressed and for dimension $d\leq 3$ existence of classical solutions
was established. In \cite{GMP}, for $g(m)=m^{\gamma}$, $0<\gamma<1$, the following a-priori estimate
was proved: $u\in W^{2,q},$ $q>1$ in dimensions $d\leq 3$.

%
%In this paper we prove various a-priori estimates for stationary mean field games which allow small dependence on velocity field of players. Using these estimates and continuity method we establish existence of smooth solutions to \eqref{main}.

The paper is structured as follows: we start in Section \ref{asec} by discussing the main hypothesis.
Then we proceed to Section \ref{elmes} where we present some elementary estimates for solutions to \eqref{main} which are analogues of the estimates for time-dependent problems in \cite{ll3} and the
ones in \cite{GMP}. In particular we prove $H^1$ bounds for $m$ and $W^{1,p}$ bounds for $u$. In Section \ref{fintp} we obtain further integrability and
regularity properties of $m$ and $u$, such as $H^1$ bounds of $|\ln m|^q$ for any $q\geq 1$, integrability of $\frac{1}{m^{r_0}}$ for some $r_0>0$, $W^{2,q}$ bounds for $u$ for some $q>1$ if $g(m)=\ln m$. Furthermore we prove $L^r$ bound for $g(m)$ with $r>d$, and $L^2$ bound for $D(g(m))$ both for logarithmic $g(m)=\ln m$ and power $g(m)=m^{\gamma}$ nonlinearities.
In Section \ref{specl} we consider Hamiltonians of a special form which can not be handled by methods of Section \ref{regadjm}. For these Hamiltonians in case of logarithmic nonlinearity $g(m)=\ln m$ we obtain $L^{\infty}$ bounds for $\frac{1}{m}$, and $W^{2,2}$ bounds for $u$. Additionally, for dimensions not greater than $3$ we establish also $W^{1,\infty}$ and $W^{3,2}$ bounds for $u$.
In the Section \ref{regadjm} we employ the adjoint method technique developed by L. C. Evans
(\cite{E3}) to prove $W^{1,\infty}$ bounds for $u$ for a broad class of Hamiltonians.
This application of the adjoint method extends the ideas in \cite{GM}.
We end the Section by proving a-priori bounds for all derivatives of $m$ and $u$ for both logarithmic $g(m)=\ln m$ and power like $g(m)=m^{\gamma}$ nonlinearities for any $\gamma>0$ if $d\leq 4$ and for $\gamma\in(0,\frac{1}{d-4})$ if $d>4$. The same bounds are also proved for the Hamiltonians of special form considered in Section \ref{specl} with logarithmic nonlinearities in dimensions not grater than $3$.
In Section \ref{exist} we use bounds from Section \ref{regadjm} and continuation method to prove existence of smooth solutions to \eqref{main}, for this we impose further assumptions which are related with the monotonicity conditions by J.M.Lasry-P.L.Lions used to establish uniqueness (see \cite{ll3},\cite{cardaliaguet}). Finally in Section \ref{exmpl} we present two examples of problems for which our existence results apply. The first example is the case without velocity field dependence, the second case concerns Hamiltonians of simple form with small dependence on velocity field.

%\newpage

\section{Assumptions}
\label{asec}

In this section we introduce and discuss
the various assumptions that will be needed throughout the paper. Further hypothesis needed for application of continuation method are discussed only in Section \ref{exist}. Other additional estimates can be proven under different Assumptions, and those are discussed only in Section \ref{specl}.
%\renewcommand{\labelenumi}{({\bf A\arabic{enumi}})}
%\setcounter{enumi}{2}
%\begin{enumerate}

We will be working under the assumption that $H$ is quasi-variational(\cite{GMP}): %\begin{hp}
\renewcommand{\labelenumi}{({\bf A\arabic{enumi}})}
\begin{enumerate}
\item\label{qzv}
There exists a function $g\colon (0,\infty)\to \mathbb{R}$ and
a continuous Hamiltonian
\[
H_0\colon\mathbb{T}^d\times\mathbb{R}^d\times \mathcal{P}(\mathbb{T}^d)\times \chi(\mathbb{T}^d)\to \mathbb{R},
\]
such that:
\begin{equation*}
|H(x,p,m,V)-H_0(x,p,m,V)+g(m(x))|\leq C,
\end{equation*}
for all $(x,p,m,V)\in \mathbb{T}^d\times\mathbb{R}^d\times \mathcal{P}^{ac}(\mathbb{T}^d)\times \chi(\mathbb{T}^d)$.

Note that unlike $H$, $H_0$ does not depend in $m$ pointwisely.
%\end{hp}
\suspend{enumerate}

\resume{enumerate}
\item
%\begin{hp}
\label{g}
%\item
The function $g:(0,\infty)\rightarrow\Rr$ is smooth, strictly increasing.
More precisely one of the following holds:
\begin{enumerate}[label=\alph*)]
\item
\label{lnm}
\[
g(m)=\ln m,
\]
\item
\label{mgama}
\[
g(m)=m^{\gamma} \text{, with } \gamma>0,
\]
\end{enumerate}
in which case we will refer to them as, respectively, Assumption (A\ref{g}\ref{lnm} or (A\ref{g}\ref{mgama}.
%\end{hp}
\suspend{enumerate}
%\begin{notee}
%In fact the previous hypothesis could be replaced by a slightly more general assumption.
%In our proofs one needs $g$ increasing, concave and $mg(m)$ convex and suitable growth conditions.
%The key differences between (A\ref{g}\ref{lnm} and (A\ref{g}\ref{mgama} is that $\ln m$ is not bounded by below and therefore different techniques are required to address this case.
%Variations of our results may be proven without assuming a specific formula for $g$ but the two cases above are illustrative of the main points and techniques.
%\end{notee}
We suppose that for the Hamiltonians $H:\mathbb{T}^d\times\mathbb{R}^d\times \mathcal{P}^{ac}(\mathbb{T}^d)\times \chi(\mathbb{T}^d)\to \mathbb{R}$ and $H_0:\mathbb{T}^d\times\mathbb{R}^\times \mathcal{P}(\mathbb{T}^d)\times \chi(\mathbb{T}^d)\to \mathbb{R}$  the following assumptions are satisfied:

%\begin{hp}
%Assume that the equation
%\begin{equation}
%\label{V}
%V(x)=D_pH(x,P(x),m,V)
%\end{equation}
%with $P\in  \chi(\mathbb{T}^d),$ can be solved for $V$ as $V=G(P,m),$ where
%\[
%G\colon \chi(\mathbb{T}^d)\times\mathcal{P}^{ac}(\mathbb{T}^d)\to \chi(\mathbb{T}^d).
%\]
%\end{hp}
\resume{enumerate}
%\begin{hp}
 \item\label{pbbl}
There exist constants
$C,\,\delta>0$ such that
for all $(x,p,m,V)\in\mathbb{T}^d\times\mathbb{R}^d\times\mathcal{P}(\mathbb{T}^d) \times \chi(\mathbb{T}^d)$,
\begin{equation*}
|p|^2\leq C+CH_0(x,p,m,V)+\delta \int_{\mathbb{T}^d} |V|^2\dm.
\end{equation*}
\item\label{difp}
For any $m\in\mathcal{P}^{ac}(\mathbb{T}^d)$ and $V\in\chi(\mathbb{T}^d)$ the function $H(x,p,m,V)+g(m(x))$ is smooth in variables $x,p$ with locally uniformly bounded derivatives.
%\end{hp}
\suspend{enumerate}
For $(x,p,m,V)\in\mathbb{T}^d\times\mathbb{R}^d\times \mathcal{P}^{ac}(\mathbb{T}^d)\times \chi(\mathbb{T}^d)$
we define the Lagrangian $L$ associated with $H$ as
\begin{equation}
\label{lagr}
L(x,p,m,V)=-H(x,p,m,V)+pD_pH(x,p,m,V).
\end{equation}
Note that if $(u,m,V)$ solves \eqref{main},  then
\[
L(x,Du,m,V)=-H(x,Du,m,V)+Du\cdot V.
\]
With this notation we assume further:
\resume{enumerate}
\item
%\begin{hp}
\label{bbl1}
There exists $c>0$ such that, for all $(x,p,m,V)\in\mathbb{T}^d\times\mathbb{R}^d\times \mathcal{P}^{ac}(\mathbb{T}^d)\times  \chi(\mathbb{T}^d)$
\begin{equation}
L(x,p,m,V)\geq cH_0(x,p,m,V)+g(m)-C-\delta \int_{\mathbb{T}^d} |V|^2\dm.
\end{equation}
%\end{hp}
\item
%\begin{hp}
\label{qbd}
For all $(x,p,m,V)\in\mathbb{T}^d\times\mathbb{R}^d\times \mathcal{P}^{ac}(\mathbb{T}^d)\times  \chi(\mathbb{T}^d)$
\begin{equation}
|D_{p}H(x,p,m,V)|^2\leq C+ CH_0(x,p,m,V) +\delta \int_{\mathbb{T}^d} |V|^2\dm.
\end{equation}
\suspend{enumerate}
%\begin{hp}\label{bbl}
%For the same $\delta$ as in Assumption \ref{pbbl} there exists constant $c,C$ such that
%\begin{equation}
%\int_{\mathbb{T}^d}L(x,\phi(x),m,V)\dm\geq -C-\delta \int_{\mathbb{T}^d}|V|^2\dm
%\end{equation}
%\end{hp}
%for any function $\phi\colon\mathbb{T}^d\to\mathbb{R}^d.$

The following hypothesis depends implicitly upon the bounds given in Proposition \ref{bV}.

\resume{enumerate}
\item
\label{delta0}
We assume
$$\delta\in[0,\delta_0],$$ where $\delta_0$ is given by Proposition \ref{bV}.
\suspend{enumerate}

Another hypothesis concerns the convexity of $H$ in $p$. We suppose:
\resume{enumerate}
\item
%\begin{hp}
\label{conv}
$H$ is uniformly convex in $p$: there exists $\kappa>0$ such that for all $(x,p,m,V)\in\mathbb{T}^d\times\mathbb{R}^d\times\mathcal{P}^{ac}(\mathbb{T}^d)\times  \chi(\mathbb{T}^d)$
\begin{equation*}
D_{pp}^2H(x,p,m,V)\geq\kappa I,
\end{equation*} where $I$ is the identity matrix in $\Rr^d$.
%\end{hp}
\suspend{enumerate}
%\end{hp}
%and
%\begin{hp}\label{Gbd}
%There exists a constant $C$ such that
%\begin{equation}
%\int_{\mathbb{T}^d}|G|^2\dm\leq C \int_{\mathbb{T^d}}|P|^2\dm +C.
%\end{equation}
%\end{hp}
Set
\begin{equation*}
\label{Hx}
\widehat{H}_{x}=D_{x}(H(x,p,m,V)+g(m(x))),
\end{equation*}
\begin{equation*}
\label{Hxx}
\widehat{H}_{xx}=D_{xx}(H(x,p,m,V)+g(m(x))),
\end{equation*}
\begin{equation}
\label{Hxp}
\widehat{H}_{xp}=D_x(D_pH(x,p,m,V))
\end{equation}
then we require
\resume{enumerate}
%\begin{hp}
\item\label{bdv}
For all $(x,p,m,V)\in\mathbb{T}^d\times\mathbb{R}^d\times \mathcal{P}^{ac}(\mathbb{T}^d)\times  \chi(\mathbb{T}^d)$
\begin{equation*}
|\widehat{H}_{xx}(x,p,m,V)|\leq C+CH_0(x,p,m,V)+\delta \int_{\mathbb{T}^d} |V|^2\dm.
\end{equation*}
%\end{hp}
\item
%\begin{hp}
\label{bdv1}
For all $(x,p,m,V)\in\mathbb{T}^d\times\mathbb{R}^d\times \mathcal{P}^{ac}(\mathbb{T}^d)\times  \chi(\mathbb{T}^d)$
\begin{equation*}
|\widehat{H}_{xp}(x,p,m,V)|^2\leq C+CH_0(x,p,m,V)+\delta \int_{\mathbb{T}^d} |V|^2\dm.
\end{equation*}

%\item

%\label{hba}
%For all $(x,p,V)\in\mathbb{T}^d\times\mathbb{R}^d\times  \chi(\mathbb{T}^d)$
%\begin{equation*}
%H_0(x,p,m, V)\leq C+C|p|^2+\delta \int_{\mathbb{T}^d} |V|^2\dm.
%\end{equation*}
%\end{hp}

\item
\label{bdv2}
There exists $0\leq \beta<2$ such that for all $(x,p,m,V)\in\mathbb{T}^d\times\mathbb{R}^d\times \mathcal{P}^{ac}(\mathbb{T}^d)\times  \chi(\mathbb{T}^d)$
\begin{equation}
|\widehat{H}_{x}(x,p,m,V)|\leq C+C|p|^\beta.
\end{equation}

%\item
%\label{solvb}
%There exists a differentiable mapping $R\colon\chi(\mathbb{T}^d)\times\mathcal{P}^{ac}(\mathbb{T}^d)\to\chi(\mathbb{T}^d)$, such that for any
%$(P,m)\in\chi(\mathbb{T}^d)\times\mathcal{P}^{ac}(\mathbb{T}^d)$, $V=R(P,m)$ is the unique solution to
%$V(x)=D_pH(x,P(x),m,V).$

\suspend{enumerate}

\section{Elementary estimates}
\label{elmes}
In this section we are going to prove various a-priori estimates for the solutions to the
quasivariational stationary extended mean-field game equation \eqref{main}. Hereafter, by a solution to \eqref{main} we mean a classical solution, with $m>0$. Later we will use these estimates to prove the existence of smooth solutions to \eqref{main} by the continuation method.
These estimates can be regarded as the analog for stationary problems as the
estimates for the time-dependent case in \cite{ll3}. Our presentation will be based upon
the ideas and techniques from \cite{GMP}
with some modifications to allow for the dependence of $H$ on $V$.

\begin{pro}
\label{hbar} Assume (A\ref{qzv})-(A\ref{bbl1}). Let
$(u,m,V,\overline{H})$ solve system \eqref{main}, then there exists  $C>0$  such that
\begin{equation*}
\left|\int_{\mathbb{T}^d}g(m)\dx\right|,
\left|\int_{\mathbb{T}^d}mg(m)\dx\right|\leq C+C\delta \int_{\mathbb{T}^d} |V|^2\dm,
\end{equation*}
and
\begin{equation}
\label{hbdd}
|\overline{H}|\leq C+C\delta \int_{\mathbb{T}^d} |V|^2\dm.
\end{equation}

\end{pro}
\begin{proof}
 Multiplying the first equation of \eqref{main} by $m$ and integrating by parts, we get:
\begin{equation}
\label{hleq}
\begin{split}
\overline{H} & =\int\limits_{\mathbb{T}^d}\Delta u  +H(x,Du,m,V)\dm \\&=\int\limits_{\mathbb{T}^d} H(x,Du,m,V) -VDu  \dm \\&= -\int\limits_{\mathbb{T}^d}L(x,Du,m,V)\dm\\&
\leq -\int\limits_{\mathbb{T}^d}\left[cH_0(x,Du,m,V)+g(m)\right]dm+\delta \int_{\mathbb{T}^d} |V|^2\dm+C\\&
\leq -\int\limits_{\mathbb{T}^d}g(m)\dm +C\delta \int_{\mathbb{T}^d} |V|^2\dm+C,
\end{split}
\end{equation}
where recall that $L$ is given by \eqref{lagr} and we used  (A\ref{bbl1}), (A\ref{pbbl}) and (A\ref{g}).

For the opposite inequality, from the first equation in  \eqref{main}, (A\ref{qzv}) and (A\ref{pbbl}) we have
$$
\overline{H}\geq \Delta u(x)-g(m)-C-\delta \int_{\mathbb{T}^d} |V|^2\dm.
$$
Integrating in $x$ we get
\begin{equation}
\label{hgeq}
 \overline{H}\geq-\int\limits_{\mathbb{T}^d}g(m)\dx
-C-\delta \int_{\mathbb{T}^d} |V|^2\dm.
\end{equation}
Combining \eqref{hleq} and \eqref{hgeq} we obtain
\[
\int\limits_{\mathbb{T}^d}mg(m)\dx\leq \int\limits_{\mathbb{T}^d}g(m)\dx+C\delta \int_{\mathbb{T}^d} |V|^2\dm+C,
\]
now by noting that for both  Assumptions (A\ref{g}\ref{lnm} and (A\ref{g}\ref{mgama} we have
$mg(m)\geq -C$ and $g(m)\leq \frac{1}{2}mg(m)+C$ for some constant $C>0,$ we conclude
\[
\left|\int_{\mathbb{T}^d}g(m)\dx\right|,
\left|\int_{\mathbb{T}^d}mg(m)\dx\right|\leq C+C\delta \int_{\mathbb{T}^d} |V|^2\dm,
\]
plugging these in \eqref{hleq} and \eqref{hgeq} yields \eqref{hbdd}.

\end{proof}

\begin{cor}
\label{hox}Assume (A\ref{qzv})-(A\ref{bbl1}). Let
$(u,m,V,\overline{H})$ solve the system \eqref{main},  then there exists  $C>0$  such that

$$\int_{\mathbb{T}^d} H_0(x,Du,m, V) \dx\leq C+ C\delta \int_{\mathbb{T}^d} |V|^2 \dm.$$
\end{cor}

\begin{proof}From the first equation in \eqref{main}, (A\ref{qzv}),  Proposition \ref{hbar}, we have
$$
\int_{\mathbb{T}^d} H_0(x,Du,m, V)\dx\leq\overline{H}+\int_{\mathbb{T}^d}g(m)\dx +C\leq C+C\delta \int_{\mathbb{T}^d} |V|^2\dm.
$$
\end{proof}

\begin{cor}
\label{h0m}
Assume (A\ref{qzv})-(A\ref{bbl1}). Let
$(u,m,V,\overline{H})$ solve the system \eqref{main},  then there exists  $C>0$  such that
\begin{equation}
\label{dubd}
\int\limits_{\mathbb{T}^d}H_0(x,Du,m, V)\dm \leq C+C\delta \int_{\mathbb{T}^d} |V|^2\dm.
\end{equation}

\end{cor}

\begin{proof} As in the proof of Proposition \ref{hbdd}, we have
$$
\overline{H}=-\int\limits_{\mathbb{T}^d}L(x,Du,m,V)\dm\leq C-\int_{\mathbb{T}^d}g(m)\dm-C\int\limits_{\mathbb{T}^d}H_0(x,Du,m, V)\dm
+\delta \int_{\mathbb{T}^d} |V|^2\dm,
$$
From this, using the bounds  from Proposition \ref{hbar} we end the proof.
\end{proof}

\begin{pro}
\label{sqrtm}
Assume (A\ref{qzv})-(A\ref{qbd}) and let
$(u,m,V,\overline{H})$ solve the system \eqref{main}. Then there exists  $C>0$  such that

\begin{equation*}
\|\sqrt{m}\|_{H^1}\leq C+C\delta \int_{\mathbb{T}^d} |V|^2\dm.
\end{equation*}

\end{pro}

\begin{proof}
 Multiplying the second equation of \eqref{main} by $\ln m$ and integrating by parts we get the estimate
$$
4\int _{\mathbb{T}^d}|D\sqrt{m}|^2\dx=\int _{\mathbb{T}^d}\frac{|Dm|^2}{m}\emph{d}x=\int _{\mathbb{T}^d}Dm D_pH\dx\leq \frac{1}{2}\int_{\mathbb{T}^d}\frac{|Dm|^2}{m}\dx+
\frac{1}{2}\int _{\mathbb{T}^d}|D_pH|^2\dm.
$$
Hence
$$\int _{\mathbb{T}^d}|D\sqrt{m}|^2\dx\leq C\int _{\mathbb{T}^d}|D_pH|^2\dm.$$
The result then follows from (A\ref{qbd}),  Corollary \ref{h0m} and $\int_{\mathbb{T}^d} m\dx=1$.
\end{proof}

\begin{pro}
\label{bV}
Assume (A\ref{qzv})-(A\ref{qbd}). Let
$(u,m,V,\overline{H})$ solve system \eqref{main}. Then, there exist $C,\,\delta_0>0$ such that  for any $\delta\in[0,\delta_0]$ we have \[\int_{\mathbb{T}^d} |V|^2\emph{d}m\leq C.\]
\end{pro}
\begin{proof}
Using the last equation of \eqref{main}, (A\ref{qbd}) and Corollary \ref{h0m}, we get
$$
\int_{\mathbb{T}^d} |V|^2\emph{d}m=\int_{\mathbb{T}^d}|D_pH|^2dm\leq C+C\int _{\mathbb{T}^d}H_0dm +\delta \int_{\mathbb{T}^d} |V|^2\emph{d}m\leq C+C\delta \int _{\mathbb{T}^d}|V|^2\emph{d}m,
$$
 For small $\delta$ this gives the result.
 %Note that \eqref{Gbd} with \eqref{bbl} imply \eqref{qbd} holds valued at the solution of \eqref{main}.
\end{proof}

Combining Proposition \ref{bV}
with Propositions \ref{hbar}, \ref{sqrtm}  and Corollaries \ref{hox}, \ref{h0m} we get:
\begin{cor}
\label{ABC} Assume (A\ref{qzv})-(A\ref{delta0}). Let
$(u,m,V,\overline{H})$ solve system \eqref{main}. Then, there exists $C>0$ such that
\begin{equation}
\label{intmgmbound}
\left|\int_{\mathbb{T}^d}g(m)\dx\right|,
\left|\int_{\mathbb{T}^d}mg(m)\dx\right|\leq C,
\end{equation}
\begin{equation}\label{Hbarbound}
|\overline{H}|\leq C,
\end{equation}
\begin{equation}\label{H_0dxbound}
 \left|\int_{\mathbb{T}^d} H_0(x,Du,m,V)\dx\right|\leq C, \quad \text{and so}\quad \|Du\|_{L^2(\Tt^d)}\leq C,
 \end{equation}
 \begin{equation}\label{H_0dmbound}
  \left|\int_{\mathbb{T}^d} H_0(x,Du,m,V)\dm\right|\leq C,
  \end{equation}
   \begin{equation}\label{sqrtmbound}
\|\sqrt{m}\|_{H^1(\mathbb{T}^d)}\leq C.
\end{equation}
\end{cor}

\begin{pro}
\label{D2u}
Assume (A\ref{qzv})-(A\ref{bdv1}) and let
$(u,m,V,\overline{H})$ solve system \eqref{main}.
 Then, there exists  $C>0$  such that
\begin{equation}
\label{g'D2u}
\int_{\mathbb{T}^d} g'(m)|Dm|^2 dx\leq C, \qquad \int_{\mathbb{T}^d}|D^2 u|^2dm\leq C.
\end{equation}
\end{pro}

\begin{proof}
Applying the operator $\Delta$ on the first equation of \eqref{main} we obtain
\[
\Delta^2u+\widehat{H}_{x_ix_i}+2\widehat{H}_{p_kx_i}(x,Du,m,V)u_{x_kx_i}+
Tr(D^2_{pp}H(x,Du,m,V)(D^2u)^2)+
\]
\[
D_pH(x,Du,m,V)D\Delta u-\div(g'(m)Dm)=0.
\]
Integrating with respect to $m$ we get
$$
\int_{\mathbb{T}^d} g'(m)|Dm|^2\dx+\frac{\kappa}{2}\int _{\mathbb{T}^d} |D^2u|^2\emph{d}m\leq\int_{\mathbb{T}^d}|\widehat{H}_{x_ix_i}|\dm+C|\widehat{H}_{xp}|^2\dm\leq C+C\delta \int_{\mathbb{T}^d} |V|^2\emph{d}m,
$$
where in the last inequality we used (A\ref{bdv}) and (A\ref{bdv1}), then the Proposition \ref{bV} finishes the proof.
\end{proof}

\begin{cor}
\label{greglog}
Assume (A\ref{qzv})-(A\ref{bdv1}) and (A\ref{g}\ref{lnm},  and let
$(u,m,V,\overline{H})$ solve system  \eqref{main}. Then, there exists  $C>0$  such that
\begin{equation}
\label{m2*log}
\int_{\mathbb{T}^d} m^{\frac{2^*}{2}}\dx\leq C,
\end{equation} and
\begin{equation}
\label{Du4log}
\int_{\mathbb{T}^d}H_0(x, Du, m, V)^2 \dm\leq C, \quad \text{and so}\quad
\int_{\mathbb{T}^d}|Du|^4\dm\leq C.
\end{equation}
\end{cor}
\begin{proof} From Corollary \ref{ABC}, estimate \eqref{sqrtmbound}, we know that $m^\frac{1}{2}\in H^1(\mathbb{T}^d)$. Sobolev's Theorem then implies \eqref{m2*log}.
Since $\frac{2^*}{2}>1$, from  \eqref{m2*log} we deduce in particular that, under (A\ref{g}\ref{lnm}
\begin{equation}\label{gm^2mlog}\int_{\mathbb{T}^d} g(m)^2dm\leq C.\end{equation}
Now, using (A\ref{qzv}) and  \eqref{Hbarbound}, we have
\[
H_0(x,Du,m, V)\leq C-\Delta u +g(m).
\]
Then by (A\ref{pbbl}), Proposition \ref{D2u} and \eqref{gm^2mlog}, we get
\[
\int_{\mathbb{T}^d}|Du|^4dm\leq C+C\int_{\mathbb{T}^d} H_0^2dm\leq C+ C \int_{\mathbb{T}^d} g(m)^2dm+C\int_{\mathbb{T}^d} |D^2u|^2\dm\leq C.
\]
\end{proof}

\begin{cor}
\label{greg}
Assume (A\ref{qzv})-(A\ref{bdv1}), and (A\ref{g}\ref{mgama}. Let
$(u,m,V,\overline{H})$ solve system \eqref{main}. Then, there exists  $C>0$  such that
\begin{equation}
\label{m2*}
\int_{\mathbb{T}^d} m^{\frac{2^*}{2}(\gamma +1)}\dx\leq C.
\end{equation}
Furthermore, if $2\gamma+1\leq{\frac{2^*}{2}(\gamma +1)}$, then
\begin{equation}
\label{Du4}
\int_{\mathbb{T}^d}H_0(x, Du, m, V)^2 \dm\leq C, \quad \text{and so}\quad
\int_{\mathbb{T}^d}|Du|^4\dm\leq C.
\end{equation}
\end{cor}

\begin{proof}
Let $f(x)=m^{\frac{\gamma +1}{2}}(x)$. From \eqref{intmgmbound} we have
\[
\int_{\mathbb{T}^d} f^2\dx= \int_{\mathbb{T}^d} mg(m)\dx\leq C.
\]
Note that
\[
\int_{\mathbb{T}^d} |Df|^2\dx=C\int_{\mathbb{T}^d} m^{\gamma-1}|Dm|^2\dx=C\int_{\mathbb{T}^d} g'(m)|Dm|^2\dx\leq C.
\]
%Then, by the Poincar\'{e} inequality
%\[
%\int_{\mathbb{T}^d} f^2\dx-\left (\int_{\mathbb{T}^d} f\dx\right )^2=\int_{\mathbb{T}^d}\left (f-\int_{\mathbb{T}^d} f \right )^2\dx
%\leq C\int_{\mathbb{T}^d} |Df|^2\dx\leq C.
%\]
Thus
\begin{equation*}
\|f\|_{H^1}\leq C.
\end{equation*}
The Sobolev inequality then implies
\[
\|f\|_{L^{2^*}}\leq C\|f\|_{H^1}\leq C,
\]
which proves \eqref{m2*}.
 In particular, if $2\gamma+1\leq{\frac{2^*}{2}(\gamma +1)}$, then
\begin{equation}\label{intgmm}\int_{\mathbb{T}^d} g(m)^2 m\dx=\int_{\mathbb{T}^d} m^{2\gamma+1}dx\leq C+C\int_{\mathbb{T}^d} m^{\frac{2^*}{2}(\gamma +1)}dx=C+C\int_{\mathbb{T}^d} f^{2^*} \dx\leq C.\end{equation}
Using (A\ref{qzv}) and  \eqref{Hbarbound}, we have
\[
H_0(x,Du,m, V)\leq C-\Delta u +g(m).
\]
Then by (A\ref{pbbl}), Proposition \ref{D2u} and \eqref{intgmm}
\[
\int_{\mathbb{T}^d}|Du|^4dm\leq C+C\int_{\mathbb{T}^d} H_0^2dm\leq C+ C \int_{\mathbb{T}^d} g(m)^2dm+C\int_{\mathbb{T}^d}
|D^2u|^2\dm\leq C.
\]
\end{proof}

\section{Additional integrability properties}
\label{fintp}
In this section we continue the study of various a-priori estimates, focusing our attention in $L^q$ estimates for $m$ as well as $\ln m$. This in
particular, see Theorem \ref{t4p6}, yields $W^{2,q}$ estimates for $u$.

\begin{pro}
\label{mintg}
Assume (A\ref{qzv})-(A\ref{bdv1}) and let
$(u,m,V,\overline{H})$ solve system \eqref{main}. Furthermore, suppose that  one of the following assumptions is satisfied:
\begin{itemize}
\item[(i)] (A\ref{g}\ref{lnm};
\item[(ii)]  (A\ref{g}\ref{mgama}, with $2\gamma+1\leq{\frac{2^*}{2}(\gamma +1)}$.
\end{itemize}

Then there exists a constant $C>0$, such that
\begin{equation*} \|m\|_{ L^q(\mathbb{T}^d)}\leq C\quad\forall\, 1<q<\infty, \quad \text{if }d\le 4,\end{equation*}
\begin{equation*} \|m\|_{ L^q(\mathbb{T}^d)}\leq C\quad\forall\, 1<q< \frac{d}{d-4},\quad\text{if }d\ge5.\end{equation*}

\end{pro}

\begin{proof}
Assumption  (A\ref{qbd}),  \eqref{Du4log}, \eqref{Du4} and Proposition \ref{bV}  imply
\begin{equation}
\label{hp4m}
\int_{\mathbb{T}^d}|D_pH|^4m\dx\leq C.
\end{equation}
Multiply the second equation of \eqref{main} by $m^r$, $r>0$, and integrate by parts:
\begin{equation*} \int_{\mathbb{T}^d} m^{r-1}|Dm|^2-m^rD_pH\cdot Dmdx=0.\end{equation*}
Then, by Young's inequality

\begin{equation*}\begin{split}  \int_{\mathbb{T}^d} m^{r-1}|Dm|^2dx &\leq \int_{\mathbb{T}^d} m^r|D_pH||Dm|dx\\ &
=\int_{\mathbb{T}^d} |D_pH|m^\frac{1}{4}|Dm|m^\frac{r-1}{2}m^\frac{2r+1}{4}dx\\&
\leq \int_{\mathbb{T}^d}\frac{1}{4}|D_pH|^4m+\frac{1}{2}m^{r-1}|Dm|^2+\frac{1}{4}m^{2r+1}dx.\end{split}\end{equation*}
Estimate \eqref{hp4m} then implies
\begin{equation}\label{intmrDm} \int_{\mathbb{T}^d} m^{r-1}|Dm|^2dx\le C+C\int_{\mathbb{T}^d} m^{2r+1}dx.\end{equation}

Remark that $m^{r-1}|Dm|^2=c_r|Dm^{\frac{r+1}{2}}|^2$. By Sobolev's Theorem, if $m^{\frac{r+1}{2}}\in H^1(\mathbb{T}^d)$ then $m^{\frac{r+1}{2}}\in L^{2^*}(\mathbb{T}^d)$ and
\begin{equation*} \left(\int_{\mathbb{T}^d} m^{\frac{2^*}{2}(r+1)}\right)^\frac{2}{2^*}\leq C\int_{\mathbb{T}^d} c_r m^{r-1}|Dm|^2+m^{r+1}dx,\end{equation*} where $2^*=\frac{2d}{d-2}.$ Then we have
\begin{equation}\label{prop1ineq} \left(\int_{\mathbb{T}^d} m^{\frac{2^*}{2}(r+1)}\right)^\frac{2}{2^*}\leq C+C\int_{\mathbb{T}^d} m^{2r+1}dx.\end{equation}

Now, if $d\le 4$ then $\frac{2^*}{2}(r+1)> 2r+1$ for any $r>0$, while if $d\geq 5$ the inequality is true if $r< \frac{2}{d-4}$. Under these assumptions on $r$, and $\int_{\mathbb{T}^d} m^\frac{2^*}{2}dx\leq C$, iterating
\eqref{prop1ineq}  we conclude that $m\in L^q$ for any $q>1$ if $d\le 4$. If $d\ge 5$ then
$$\int_{\mathbb{T}^d} m^{\frac{2^*}{2}(r+1)}dx\le C$$ for any $r< \frac{2}{d-4}$, i.e., $m\in L^q$ for any $1<q< \frac{d}{d-4}$.

\end{proof}

\begin{pro}\label{lminh1} Assume  (A\ref{qzv})-(A\ref{bdv1}).
Let
$(u,m,V,\overline{H})$ solve system \eqref{main}. Then
\[
\int_{\mathbb{T}^d}|D\ln m|^2\dx\leq C.
\]
Furthermore if (A\ref{g}\ref{lnm} holds
then there exists a constant $C>0,$ such that
$\|\ln m\|_{ H^1(\mathbb{T}^d)}\leq C.$
\end{pro}

\begin{proof}
Multiplying the second equation in \eqref{main} by $\frac{1}{m}$ and integrating by parts as in the proof of Proposition \ref{sqrtm} we get
\[
\int_{\mathbb{T}^d}|D\ln m|^2\dx\leq C\int_{\mathbb{T}^d}|D_pH(Du,x,m,V)|^2\dx\leq C+C\int_{\mathbb{T}^d}H_0\dx\leq C,
\]
using Assumption  (A\ref{qbd}), and \eqref{H_0dxbound}.

Now assume (A\ref{g}\ref{lnm} holds.
Integrating the first equation in \eqref{main},  using the Jensen's inequality, (A\ref{qzv})  and Corollary \ref{ABC}, we have

\[
0\geq\int_{\mathbb{T}^d}\ln m\dx\geq-C+\int_{\mathbb{T}^d} H_0(Du,x,m,V)\dx\geq-C.
\]
Therefore
\[
\left|\int_{\mathbb{T}^d}\ln m\dx\right|\leq C.
\]
Then by the Poincar\'{e} inequality
\[
\int_{\mathbb{T}^d}|\ln m|^2\dx\leq\left(\int_{\mathbb{T}^d}\ln m\dx\right)^2+\int_{\mathbb{T}^d}|D\ln m|^2\dx\leq C.
\]

\end{proof}

The previous proposition can in fact be improved as we show next.

\begin{pro}
\label{lnmreg}
Assume  (A\ref{qzv})-(A\ref{bdv1})  and (A\ref{g}\ref{lnm}.   Let
$(u,m,V,\overline{H})$ solve system \eqref{main}. Then, for every $1\leq p<\infty$,  there exists a constant $C_p>0$ such that
\[
\||\ln m|^p\|_{ H^1(\mathbb{T}^d)}\leq C_p.
\]
\end{pro}

\begin{proof}
We prove by induction that $f_k=|\ln m|^{\frac{k+1}{2}}\in H^1(\mathbb{T}^d)$, for any $k\in\Nn$. For $k=1$ this is  Proposition \ref{lminh1}.  Let $l\ge 1$  and suppose that
$\|f_k\|_{H^1(\mathbb{T}^d)}\leq C_l$ for all $k\leq l$, then we have
$$\|Df_k\|_{L^2}^2=\int_{\mathbb{T}^d}\frac{|\ln m|^{k-1}}{m^2}|Dm|^2\dx\leq C_l^2,$$
and
$$ \|f_k\|_{L^2}^2=\int_{\mathbb{T}^d}|\ln m|^{k+1}\dx\leq C_l^2.$$
We want to show that $f_{l+1}\in H^1(\mathbb{T}^d).$
Let $F_l(z)=\int_1^z\frac{|\ln y|^{l}}{y^2}\dy$ multiplying the second equation of \eqref{main} by $F_l(m)$ and integrating by parts we get
\[
\int_{\mathbb{T}^d}\frac{|\ln m|^{l}}{m^2}|Dm|^2\dx=\int_{\mathbb{T}^d}\frac{|\ln m|^{l}}{m}DmD_pH\dx\leq\frac{1}{2}\int_{\mathbb{T}^d}\frac{|\ln m|^{l}}{m^2}|Dm|^2\dx+\frac{1}{2}\int_{\mathbb{T}^d}|\ln m|^l|D_pH|^2\dx.
\]
Thus,
\begin{equation}
\label{dlnm}
\int_{\mathbb{T}^d}\frac{|\ln m|^l}{m^2}|Dm|^2\dx\leq\int_{\mathbb{T}^d}|\ln m|^l|D_pH|^2\dx\leq C\int_{\mathbb{T}^d}|\ln m|^l\dx+C\int_{\mathbb{T}^d}|\ln m|^lH_0\dx
\end{equation}
where at the last inequality we used (A\ref{qbd}) and Proposition \ref{bV}.

Next, from the first equation of \eqref{main},(A\ref{qzv}) and \eqref{Hbarbound}, we infer that
\[
H_0(x,Du,m,V)\leq C+|\ln m| -\Delta u.
\]
Then, multiplying by $|\ln m|^l$ and integrating
\[
\int_{\mathbb{T}^d}|\ln m|^lH_0\dx\leq C\int_{\mathbb{T}^d}|\ln m|^l\dx+ C \int_{\mathbb{T}^d}|\ln m|^{l+1}\dx -\int_{\mathbb{T}^d}\Delta u|\ln m|^l\dx.
\]
Integrating by parts the last term
\[
\int_{\mathbb{T}^d}|\ln m|^lH_0\dx\leq C\int_{\mathbb{T}^d}|\ln m|^l\dx+ C \int_{\mathbb{T}^d}|\ln m|^{l+1}\dx +\int_{\mathbb{T}^d}Du|\ln m|^{l-1}\frac{Dm}{m} sgn{(\ln m)}\dx.
\]
The integration by parts is justified, we just observe the that for a smooth function $f$ the identity $D(|f|^p)=p|f|^{p-2}sgn(f)Df$ holds both a.e. and in distribution sense.

Then,  using (A\ref{pbbl})

\begin{align*}
\int_{\mathbb{T}^d}|\ln m|^lH_0\dx&\leq C\int_{\mathbb{T}^d}|\ln m|^l\dx+ C \int_{\mathbb{T}^d}|\ln m|^{l+1}\dx+ C \int_{\mathbb{T}^d}|\ln m|^{l-1}\dx\\ &+C\int_{\mathbb{T}^d}|\ln m|^{l-1}H_0\dx+C\int_{\mathbb{T}^d}|\ln m|^{l-1}\frac{|Dm|^2}{m^2}\dx\\&
 \leq C+C \int_{\mathbb{T}^d}|\ln m|^{l+1}\dx\\ &+C\int_{\mathbb{T}^d}[\ep|\ln m|^{l}+C(\ep)]H_0\dx+C\int_{\mathbb{T}^d}|\ln m|^{l-1}\frac{|Dm|^2}{m^2}\dx,
\end{align*}
which yields
\begin{equation}
\label{hlnm}
\int_{\mathbb{T}^d}|\ln m|^lH_0\dx\leq  C \int_{\mathbb{T}^d}|\ln m|^{l+1}\dx +C\int_{\mathbb{T}^d}|\ln m|^{l-1}\frac{|Dm|^2}{m^2}\dx+C.
\end{equation}
Combining \eqref{dlnm} and \eqref{hlnm} we get
\[
\int_{\mathbb{T}^d}\frac{|\ln m|^l}{m^2}|Dm|^2\dx\leq  C \int_{\mathbb{T}^d}|\ln m|^{l+1}\dx +C\int_{\mathbb{T}^d}|\ln m|^{l-1}\frac{|Dm|^2}{m^2}\dx+C,
\]
that is,
\[
\|Df_{l+1}\|^2_{L^2}\leq C\|f_{l}\|^2_{L^2}+C\|Df_{l}\|^2_{L^2}+C\leq C_{l+1}.
\]
Since $|\ln m|^{l+1}=f_{l}^2\in{L^1}$ we have $f_{l+1}=|\ln m|^{\frac{l}{2}+1}\in L^1,$ then by the Poincar\'{e} inequality
\[
\|f_{l+1}\|^2_{L^2}\leq \|f_{l+1}\|_{L^1}^2+C\|Df_{l+1}\|^2_{ L^2}\leq C_{l+1},
\]
and this concludes the proof.
\end{proof}

\begin{cor}
\label{gmint}
Assume (A\ref{qzv})-(A\ref{bdv1}) and let
$(u,m,V,\overline{H})$ solve system \eqref{main}. Furthermore, suppose that  one of the following assumptions is satisfied:
\begin{itemize}
\item[(i)] (A\ref{g}\ref{lnm}
\item[(ii)]  (A\ref{g}\ref{mgama}, with any $\gamma>0$ if $d\leq 4,$ and $\gamma<\frac{1}{d-4}$ if $d\geq 5.$
\end{itemize}
Then there exist a constant $C>0$ and an exponent $r>d$ such that $\|g(m)\|_{L^r}$,
$\|D(g(m))\|_{L^2}\leq C.$
\end{cor}

\begin{proof}
The case (i) is a direct consequence of the Proposition \ref{lnmreg}. For the case (ii), note that the condition
 $2\gamma+1\leq{\frac{2^*}{2}(\gamma +1)}$ is satisfied then Proposition \ref{mintg} implies that $\|g(m)\|_{L^r}\leq C.$

To prove that $\|D(g(m))\|_{L^2}\leq C$ under (A\ref{g}\ref{mgama},
observe that
 \[
 \int_{\mathbb{T}^d}|D(g(m))|^2\dx=\gamma^2\int_{\mathbb{T}^d}m^{2\gamma-2}|Dm|^2\dx.
 \]
%Note that for $d\geq 5$, $0\leq \gamma <\frac{1}{d-4}\leq 1$
If $0<\gamma\leq 1$ then
there exists a constant $C$ such that
\[
m^{2\gamma-2}\leq C m^{-2}+Cm^{\gamma-1}.
\]
Therefore
\[
\int_{\mathbb{T}^d}|D(g(m))|^2\dx=C\int_{\mathbb{T}^d}m^{-2}|Dm|^2\dx+C\int_{\mathbb{T}^d}m^{\gamma-1}|Dm|^2\dx.
\]
The boundedness of the first term on the right hand side
follows from Proposition \ref{lminh1}, whereas the second term is controlled thanks to Proposition  \ref{D2u}.

If $\gamma>1$, from \eqref{intmrDm} in the proof of  Proposition \ref{mintg} and  the first estimate in the same Proposition we have
\[
\int_{\mathbb{T}^d} m^{r-1}|Dm|^2dx\le C+C\int_{\mathbb{T}^d} m^{2r+1}dx\leq C_r
\]
for any $r>0$, taking $r=2\gamma-1$ we get $\|D(g(m))\|_{L^2}\leq C$.
\end{proof}
\begin{pro}\label{1/mintegrabgeneralprop} Assume (A\ref{qzv})-(A\ref{bdv1})  and (A\ref{g}\ref{lnm}. Then there exists $r_0>0$ such that
\begin{equation}\label{h/mrintegr}\int_{\mathbb{T}^d} \frac{H_0}{m^{r_0}}dx\leq C,\end{equation} and
\begin{equation}\label{1/mrintegr}\int_{\mathbb{T}^d} \frac{1}{m^{r_0}}dx\leq C.\end{equation}
\end{pro}

\proof
Multiplying  the first equation in  \eqref{main} by $\frac{1}{m^r}$, $r>0$,  integrating by parts and using  (A\ref{qzv}) and \eqref{Hbarbound}, we get
\begin{equation}
\label{h0mr}
\int_{\mathbb{T}^d}\frac{H_0}{m^r}dx\le\int_{\mathbb{T}^d} -r\frac{Du\cdot Dm}{m^{r+1}}+\frac{\ln m+C}{m^r}dx.\end{equation}
Next, multiplying  the second equation in  \eqref{main}  by $\frac{1}{m^{r+1}}$ and integrating by parts, we obtain
\begin{equation}
\label{dmmr}
\int_{\mathbb{T}^d}\frac{|Dm|^2}{m^{r+2}}dx=\int_{\mathbb{T}^d}\frac{D_p H\cdot Dm}{m^{r+1}}dx.
\end{equation}
Let us sum the equation \eqref{h0mr} and equation \eqref{dmmr} multiplied by $r$:
\begin{equation*}\begin{split} \int_{\mathbb{T}^d}\frac{H_0}{m^r}dx+r\int_{\mathbb{T}^d}\frac{|Dm|^2}{m^{r+2}}dx&\le\int_{\mathbb{T}^d} r\frac{(D_p H-Du)\cdot Dm}{m^{r+1}}+\frac{\ln m+C}{m^r}dx\\&
\le \int_{\mathbb{T}^d} Cr\frac{(|D_p H|+|Du|)|Dm|}{m^{r+1}}+\frac{\ln m+C}{m^r}dx\\&
\le \int_{\mathbb{T}^d} \frac{1}{2}\frac{H_0}{m^r}+Cr^2\frac{|Dm|^2}{m^{r+2}}+\frac{\ln m+C}{m^r}dx,
\end{split}\end{equation*} where we used  (A\ref{pbbl}) and (A\ref{qbd}).
Now, let $r_0>0$ small enough such that $r_0\leq  Cr_0^2$, then
%\begin{equation*} \int_{\mathbb{T}^d}\frac{1}{2}C_0\frac{|Du|^2}{m^r}dx+r\int_{\mathbb{T}^d}\frac{|Dm|^2}{m^{r+2}}dx
\begin{equation*} \int_{\mathbb{T}^d} \frac{1}{2}\frac{H_0}{m^{r_0}}dx\le\int_{\mathbb{T}^d} \frac{\ln m+C}{m^{r_0}}dx\le C.\end{equation*}
and \eqref{h/mrintegr} is proven. Moreover, the previous inequalities and (A\ref{pbbl}) imply that
$$0\leq \int_{\mathbb{T}^d} \frac{\ln m+C}{m^{r_0}}dx\le C,$$ from which
 \eqref{1/mrintegr} follows.
\finedim

\begin{teo}
\label{t4p6}
Assume (A\ref{qzv})-(A\ref{bdv1})  and (A\ref{g}\ref{lnm}, then there exists $q>0$ and a constant $C>0$ such that
$$\|u\|_{ W^{2,1+q}(\mathbb{T}^d)}\leq C.$$
\end{teo}
\proof

Raise the inequality
$$|\Delta u|\le H_0(x,Du,m,V)+|\ln m|+C,$$ to the power $q+1$ with $0<q\le\frac{1}{2}$, and integrate:

\begin{equation}\label{thmw2preggener}\int_{\mathbb{T}^d} |\Delta u|^{q+1}dx\leq C_q\int_{\mathbb{T}^d} (H_0^{q+1}+|\ln m|^{q+1})dx+ C_q.\end{equation}
Next from the Young's inequality,  \eqref{Du4log}  and \eqref{H_0dxbound}, for any $0\le r\le1$, we have
\begin{equation*} \int_{\mathbb{T}^d} H_0^{1+r}m^rdx=\int_{\mathbb{T}^d} H_0^{2r}m^r H^{1-r}_0dx\le\int_{\mathbb{T}^d} r H_0^2m+(1-r)H_0dx\le C.\end{equation*}
Hence, we can estimate $\int_{\mathbb{T}^d} H_0^{q+1}dx$ as follows
\begin{equation}\label{thmw2preggener2} \int_{\mathbb{T}^d} H_0^{q+1}dx=\int_{\mathbb{T}^d} H_0^{q+\frac{1}{2}}m^q\frac{H_0^{\frac{1}{2}}}{m^q}dx\le \frac{1}{2}\int_{\mathbb{T}^d} H_0^{2q+1}m^{2q}+\frac{H_0}{m^{2q}}dx\leq C,\end{equation}
if $q\le \frac{r_0}{2}$, where $r_0$ is given by Proposition \ref{1/mintegrabgeneralprop}. Hence, from \eqref{thmw2preggener},  \eqref{thmw2preggener2} and Proposition
\ref{lnmreg},  there exists a $q>0$ such that $\|\Delta u\|_{L^{1+q}(\mathbb{T}^d)}\leq C$. The elliptic theory then implies that $\|u\|_{W^{2,1+q}(\mathbb{T}^d)}\leq C$.
\finedim

%%%%%%%%%%%%%%%%%%%%%%%%%%%%%%%%%%%%%%%%%%%%%%%%%%%%%%%%%%%%%%%%
%%%%%%%%%%%%%%%%%%%%%%%%%%%%%%%%%%%%%%%%%%%%%%%%%%%%%%%%%%%%%%%%

\section{Further estimates for special Hamiltonians}
\label{specl}
In this section we consider the equation \eqref{main} for a special class of Hamiltonians. We assume that $H$ satisfies the following hypothesis:
\begin{itemize}
\item[{\bf(H1)}]
There exist a function
\[
G\colon\mathbb{T}^d\times\mathbb{R}^d\times \mathcal{P}^{ac}(\mathbb{T}^d)\times \chi(\mathbb{T}^d)\to\Rr
\]
with \[|G(x,0,m,V)|,|D_pG(x,p,m,V)|^2, |D_xG(x,p,m,V)|\leq C+\epsilon\int_{\mathbb{T}^d} |V|^2\emph{d}m\] for some constants $C,\epsilon>0$,
and a twice continuously differentiable function
\[
\alpha\colon
\mathbb{T}^d\to\Rr,\quad\text{with }\alpha\geq\alpha_0>0,
\]
such that
\[
H(x,p,m,V)=H_0(x,p,m,V)-g(m),
\]
and
\[
H_0(x,p,m,V)=\alpha(x)\frac{|p|^2}{2}+G(x,p,m,V).
\]
%We assume further that
\item [{\bf(H2)}]
$G$ is twice differentiable in $x, p$ with
\[
|D^2_{xp}G(x,p,m,V)|^2, |D^2_{xx}G(x,p,m,V)|, |D^2_{pp}G(x,p,m,V)|\leq C+\epsilon\int_{\mathbb{T}^d} |V|^2\emph{d}m,
\]
and additionally there exists $\kappa>0$ such that
\[
D^2_{pp}H(x,p,m,V)\geq \kappa I.
\]
\end{itemize}
It is easy to check that there exists a constant $\epsilon_0>0$ such that if (H1) holds true for $\epsilon\in[0,\epsilon_0]$, then  $H$ satisfies the Assumptions (A\ref{qzv})-(A\ref{delta0}), if further (H2) holds then $H$ also satisfies (A\ref{conv})-(A\ref{bdv1}).

\begin{teo}\label{1/mintegrab}Assume (H1) with $\epsilon\in[0,\epsilon_0]$ and  (A\ref{g}\ref{lnm}. Then for any solution $(u,m,V,\Hh)$ to \eqref{main}
\begin{equation*}\left\|\frac{1}{m}\right\|_{ L^\infty(\mathbb{T}^d)}\leq C.\end{equation*}
\end{teo}
\begin{proof}
First note that from Proposition \ref{bV} we have $\int_{\mathbb{T}^d} |V|^2\emph{d}m\leq C$, thus $|D_pG(x,Du,m,V)|\leq C$ and $H_0(x,p,m,V)\ge C_0|p|^2-C$ for some constants $C_0,C>0$.
Now multiply the first equation in \eqref{main} by $\frac{\alpha(x)}{m^r}$ and integrate by parts:
\begin{equation*}\begin{split} \int_{\mathbb{T}^d} \frac{\alpha(x)}{m^r}H_0dx&=\int_{\mathbb{T}^d} \frac{Du\cdot D\alpha}{m^r} -r\alpha(x)\frac{Du\cdot Dm}{m^{r+1}}+\frac{\alpha(x)}{m^r}(\ln m+\overline{H}) dx\\&
\le \int_{\mathbb{T}^d} \frac{\alpha_0C_0}{2}\frac{|Du|^2}{m^r}+\frac{C}{m^r}-r\alpha(x)\frac{Du\cdot Dm}{m^{r+1}}+\frac{\alpha(x)}{m^r}(\ln m+\overline{H}) dx.
\end{split} \end{equation*}
Then, using again the properties of  $H_0$ and $\alpha$ and \eqref{Hbarbound}, we get
\begin{equation}\label{prop3inew1} \frac{\alpha_0C_0}{2}\int_{\mathbb{T}^d} \frac{|Du|^2}{m^r}dx\leq \int_{\mathbb{T}^d} -r\alpha(x)\frac{Du\cdot Dm}{m^{r+1}}+\frac{\alpha(x)\ln m+C}{m^r}dx.\end{equation}
Next, multiply the second equation in \eqref{main} by $\frac{1}{m^{r+1}}$:
\begin{equation*} \int_{\mathbb{T}^d} \frac{|Dm|^2}{m^{r+2}}-\frac{D_pH_0\cdot Dm}{m^{r+1}}dx=0.\end{equation*} Using the expression of $H_0$, we find
\begin{equation*} \int_{\mathbb{T}^d} \alpha(x)\frac{Du\cdot Dm}{m^{r+1}}dx=\int_{\mathbb{T}^d} \frac{|Dm|^2}{m^{r+2}}-\frac{D_pG\cdot Dm}{m^{r+1}}dx.\end{equation*} Substituting this expression in \eqref{prop3inew1}, we get
\begin{equation*}\begin{split} C\int_{\mathbb{T}^d} \frac{|Du|^2}{m^r}dx+r\int_{\mathbb{T}^d} \frac{|Dm|^2}{m^{r+2}}dx&\leq \int_{\mathbb{T}^d} r\frac{D_pG\cdot Dm}{m^{r+1}}+\frac{\alpha(x)\ln m+C}{m^r}dx\\&
\leq \int_{\mathbb{T}^d} \frac{r}{2} \frac{|Dm|^2}{m^{r+2}}+\frac{\alpha(x)\ln m+C(r+1)}{m^r}dx.\end{split}\end{equation*}
We conclude that
\begin{equation}\label{prop3inew2} C\int_{\mathbb{T}^d} \frac{|Du|^2}{m^r}dx+ \frac{r}{2}\int_{\mathbb{T}^d} \frac{|Dm|^2}{m^{r+2}}dx\le \int_{\mathbb{T}^d} \frac{\alpha(x)\ln m+C(r+1)}{m^r}dx.\end{equation}

%Hence
%$$\int_{\mathbb{T}^d} -\frac{\alpha(x)\ln m}{m^r}dx\leq \int_{\mathbb{T}^d} \frac{C(r+1)}{m^r}dx.$$ On the other hand, since $\alpha(x)\ge\alpha_0>0$, for any $r>0$ there exists $C_r>0$ such that
%$$\int_{\mathbb{T}^d} \frac{2C(r+1)}{m^r}dx\leq \int_{\mathbb{T}^d} -\frac{\alpha(x)\ln m}{m^r}dx+C_r.$$ We conclude that
% $$\int_{\mathbb{T}^d} \frac{1}{m^r}dx\leq \tilde{C}_r$$ for any $r>0$.

On the other hand, since $\alpha(x)\ge\alpha_0>0$, for any $r>0$ there exists $C_r>0$ such that
$$\int_{\mathbb{T}^d} \frac{2C(r+1)}{m^r}dx\leq \int_{\mathbb{T}^d} -\frac{\alpha(x)\ln m}{m^r}dx+C_r.$$ We conclude that
 $$\int_{\mathbb{T}^d} \frac{1}{m^r}dx\leq \tilde{C}_r$$ for any $r>0$.

 Next, we have
 \begin{equation}\label{thmi/mlinfi3}\begin{split}  \int_{\mathbb{T}^d} \frac{\alpha(x)\ln m+C(r+1)}{m^r}dx&\le \int_{\{m\ge 1\}} \frac{\alpha(x)\ln m}{m^r}dx+\int_{\mathbb{T}^d} \frac{C(r+1)}{m^r}dx\\&
 \le C\int_{\mathbb{T}^d} \frac{C_\delta}{m^{r-\delta}}+ \frac{r+1}{m^r}dx,
 \end{split}\end{equation}
 for any $\delta>0$. Hence, from the Sobolev inequality,  \eqref{prop3inew2} and \eqref{thmi/mlinfi3}, for any $r>0$
 \begin{equation*}\begin{split} \left(\int_{\mathbb{T}^d} \frac{1}{m^{\frac{2^*}{2}r}}dx\right)^\frac{2}{2^*}&\le C\int_{\mathbb{T}^d} \left|D \frac{1}{m^\frac{r}{2}}\right|^2+\frac{1}{m^r}dx\\&
 \le Cr\int_{\mathbb{T}^d} \frac{C_\delta}{m^{r-\delta}}+ \frac{r+1}{m^r}dx+C\int_{\mathbb{T}^d} \frac{1}{m^r}dx.
  \end{split}\end{equation*}
  Now, set $\beta:=\sqrt{\frac{2^*}{2}}>1$ and $\delta:=\frac{1}{\beta'}$, where $\beta'$ is the conjugate exponent of $\beta$. Then, we have
  \begin{equation*} \int_{\mathbb{T}^d}\frac{1}{m^r}dx\leq  \left(\int_{\mathbb{T}^d} \frac{1}{m^{\beta r}}dx\right)^\frac{1}{\beta},\end{equation*}
  and
  \begin{equation*}\begin{split} \int_{\mathbb{T}^d} \frac{1}{m^{r-\delta}}dx&\le  \left(\int_{\mathbb{T}^d} \frac{1}{m^{\beta r}}dx\right)^\frac{1}{\beta}\left(\int_{\mathbb{T}^d} m^{\delta\beta'}dx\right)^\frac{1}{\beta'}\\&=
   \left(\int_{\mathbb{T}^d} \frac{1}{m^{\beta r}}dx\right)^\frac{1}{\beta}\left(\int_{\mathbb{T}^d} mdx\right)^\frac{1}{\beta'}\\&=\left(\int_{\mathbb{T}^d} \frac{1}{m^{\beta r}}dx\right)^\frac{1}{\beta}. \end{split}\end{equation*}
 We conclude that
 \begin{equation*}  \left(\int_{\mathbb{T}^d} \frac{1}{m^{\beta^2 r}}dx\right)^\frac{1}{\beta^2}\leq C(r^2+1) \left(\int_{\mathbb{T}^d} \frac{1}{m^{\beta r}}dx\right)^\frac{1}{\beta},\end{equation*}
i.e.,
 \begin{equation*} \left\|\frac{1}{m}\right\|_{L^{\beta^2 r}(\mathbb{T}^d)}\leq C^{\frac{1}{r}} (r^2+1)^{\frac{1}{r}}\left\|\frac{1}{m}\right\|_{L^{\beta r}(\mathbb{T}^d)}.\end{equation*}
Taking $r=\beta^{k-1}$ for an integer $k>0$ we get
 \begin{equation*}\left\|\frac{1}{m}\right\|_{L^{\beta^{k+1}}(\mathbb{T}^d)}\leq C^{\frac{1}{\beta^{k-1}}} \beta^{\frac{2(k-1)}{\beta^{k-1}}}\left\|\frac{1}{m}\right\|_{L^{\beta^{k}}(\mathbb{T}^d)}.\end{equation*}
 Thus

 \begin{equation*} \left\|\frac{1}{m}\right\|_{L^{\beta^{k+1}}(\mathbb{T}^d)}\leq C^{\sum\limits_1^\infty\frac{1}{\beta^{k-1}}} \beta^{\sum\limits_1^\infty\frac{2(k-1)}{\beta^{k-1}}}\left\|\frac{1}{m}\right\|_{L^{1}(\mathbb{T}^d)}\leq C\left\|\frac{1}{m}\right\|_{L^{1}(\mathbb{T}^d)}\leq C.\end{equation*}
Sending $k\to\infty$ we infer that $\|\frac{1}{m}\|_{ L^\infty(\mathbb{T}^d)}\leq C.$ See also \cite{E1} for a similar argument.
  \end{proof}

  \begin{cor}
  \label{uw2p}Assume (H1),(H2) with $\epsilon\in[0,\epsilon_0]$ and  (A\ref{g}\ref{lnm}. Then for any solution $(u,m,V,\Hh)$ to \eqref{main} there exists a constant $C$ such that $\|u\|_{ W^{2,2}(\mathbb{T}^d)}\leq C.$
 Furthermore,
  if $d\leq 3$, for any $q\ge 1$ there exists a constant $C_q$ such that
  $\|u\|_{ W^{2,q}(\mathbb{T}^d)}\leq C_q.$
  \end{cor}
  \begin{proof} By Theorem \ref{1/mintegrab}, there exists $\bar{m}>0$ such that $m\ge \bar{m}$ in $\mathbb{T}^d$.  This and Proposition \ref{D2u} imply that
  $$\int_{\mathbb{T}^d} |D^2u|^2dx\le \frac{1}{\bar{m}}\int_{\mathbb{T}^d}  |D^2u|^2 m dx\le C.$$ Since in addition,
  $$\int_{\mathbb{T}^d} u^2dx,\quad\int_{\mathbb{T}^d} |Du|^2dx\le C,$$ we obtain that $\|u\|_{W^{2,2}(\mathbb{T}^d)}\leq C$.
  Because $m$ is bounded by below and $\|m\|_{L^1(\mathbb{T}^d)}=1$ we have $\ln m\in L^q(\mathbb{T}^d)$, for any $q>1$.
  If $d=2$, then Sobolev inequalities imply $\|Du\|_{L^q(\mathbb{T}^d)}\leq C_q$ for any $q>1$ thus from the first equation of \eqref{main} we conclude $\|\Delta u\|_{L^q(\mathbb{T}^d)}\leq C_q$ for any $q>1$.
  If $d=3$ then $2^*=6$, thus $\|Du\|_{L^{6}(\mathbb{T}^d)}\leq C$, hence the first equation of \eqref{main} implies that
 $\|\Delta u\|_{L^{3}(\mathbb{T}^d)}\leq C$. This together with Sobolev inequalities yield $\|Du\|_{L^q(\mathbb{T}^d)}\leq C_q$ for any $q>1$, using the equation again, we conclude that $\|\Delta u\|_{L^q(\mathbb{T}^d)}\leq C_q$ for any $q>1$.
  \end{proof}
		
\begin{cor}
\label{ulpsc}
Assume (H1),(H2) with $\epsilon\in[0,\epsilon_0]$ and  (A\ref{g}\ref{lnm}, then if $d\leq 3$ for any solution $(u,m,V,\Hh)$ to \eqref{main}, there exists a constant $C>0$ such that $\|u\|_{ W^{1,\infty}(\mathbb{T}^d)}\leq C$
  and $\|u\|_{ W^{3,2}(\mathbb{T}^d)}\leq C.$
\end{cor}

\begin{proof}
The first inequality follows directly from Corollary \ref{uw2p} and Sobolev inequalities if we take $q>d$
in Corollary \ref{uw2p}.
To prove the second inequality we differentiate the first equation in \eqref{main}:
\[
D(\Delta u)=-D\alpha\frac{|Du|^2}{2}-\alpha D^2uDu-D^2uD_pG-G_x+D(\ln m)
\]
then Corollaries \ref{uw2p} and \ref{lminh1} imply $\|D(\Delta u)\|_{L^2}\leq C$ for some constant $C>0,$ thus
$\|u\|_{ W^{3,2}(\mathbb{T}^d)}\leq C.$

\end{proof}
Using these estimates we will prove further regularity estimates for this case in Section \ref{adjointmethsec}  (see Theorem \ref{usmooth}).

%  \finedim

%%%%%%%%%%%%%%%%%%%%%%%%%%%%%%%%%%%%%%%%%
%%%%%%%%%%%%%%%%%%%%%%%%%%%%%%%%%%%%%%%%%%

\section{Improved regularity by the adjoint method}\label{adjointmethsec}
\label{regadjm}
In this section we use adjoint method techniques to prove higher regularity estimates for the solutions to \eqref{main}. For later convenience we discuss a more general situation.
\renewcommand{\labelenumi}{({\bf R\arabic{enumi}})}
\begin{enumerate}
\item Let $F\colon \mathbb{T}^d\times{\Rr^d}\to \Rr$ be a function which satisfies for some constants $c,C>0$:

\label{r1}
\begin{equation*}
|D_pF(x,p)|^2\leq C|p|^2+C.
\end{equation*}
\item
\label{r2}
\begin{equation*}
D_pF(x,p)p-F(x,p)\geq c|p|^2+\zeta(x),
\end{equation*}
with
\begin{equation*}\|\zeta\|_{L^r}\leq C,\quad\text{for some }r>d.\end{equation*}
\item
\label{r3}
Let $\widehat{F}_x(x,p)=D_x(F(x,p)+\zeta(x)),$ then
\begin{equation*}
|\widehat{F}_x(x,p)|\leq C+\psi(x)|p|^\beta
\end{equation*} with $$0\leq\beta<2,\quad \psi\in L^\frac{2r}{2-\beta}(\mathbb{T}^d).$$
\suspend{enumerate}
Consider the equation
\begin{equation}
\label{main2}
\Delta w+F(x,Dw)=0.
\end{equation}
\resume{enumerate}
\item
\label{r4}
We suppose that for any solution to \eqref{main2} we have the following a-priori bound:
\begin{equation*}
\|Dw\|_{L^2(\mathbb{T}^d)}\leq C.
\end{equation*}
\end{enumerate}
Note that $w$ solves the time dependent equation

\begin{equation}
\label{ut}
w_t+\Delta w+F(x,Dw)=0.
\end{equation}
For any $x_0\in \mathbb{T}^d$, we introduce the adjoint  variable $\rho$ as the solution of

\begin{equation}
\label{rho}
\begin{cases}
\rho_t+\div(D_pF(x,Dw(x))\rho)=
\Delta\rho,\\
\rho(x,0)=\delta_{x_0}.
\end{cases}
\end{equation}
By the maximum principle $\rho\geq 0$. Furthermore by integrating  the equation we get $\frac{d}{dt}\int_{\mathbb{T}^d}\rho(x,t)\dx=0$. In particular, for any $t>0$
\begin{equation}\label{rhodensityadjsec}\int_{ \mathbb{T}^d}\rho(x,t)dx=1.\end{equation}

\begin{pro}Assume (R\ref{r1})-(R\ref{r4}). Let $w$ and $\rho$ solve \eqref{main2} and \eqref{rho} respectively. Then, for any $T>0$
\label{wx0}
\begin{align*}
w(x_0) = &\int_0^T\int_{\mathbb{T}^d}(F(x,Dw)-Dw\cdot D_pF(x,Dw))\rho(x,t)\dx\dt +\int_{\mathbb{T}^d}w(x)\rho(x,T)\dx.
\end{align*}

\end{pro}

\begin{proof}
We just multiply equation \eqref{ut} by $\rho$ and integrate by parts using the equation for $\rho$.
\end{proof}
For fixed $T>0$, let us denote

\[
\|\rho\|_{L^1(L^q(dx),dt)}=\int_0^T\|\rho(.,t)\|_{L^q(\mathbb{T}^d)}dt.
\]

Denote by $\osc(f)=\sup_xf-\inf_x f$, for any bounded function $f\colon\mathbb{T}^d\to\Rr$.
Then we have
\begin{cor}
\label{dw2ro}
Assume (R\ref{r1})-(R\ref{r4}). Let $w$ and $\rho$ solve  \eqref{main2} and \eqref{rho} respectively. Then
\[
\int_0^T\int_{\mathbb{T}^d}|Dw|^2\rho(x,t)dxdt\leq C\|\rho\|_{L^1(L^q(dx),dt)}+C \osc(w),
\]
where $q$ is the conjugate exponent of $r$ defined by $\frac{1}{r}+\frac{1}{q}=1.$
\end{cor}
\begin{proof}
We use (R\ref{r2}) and Proposition \ref{wx0} to get

\[
\int_0^T\int_{\mathbb{T}^d}|Dw|^2\rho(x,t)dxdt\leq C \osc(w)-C\int_0^T\int_{\mathbb{T}^d}\zeta(x)\rho(x,t) \dx \dt.
\]
Now, using H\"{o}lder's inequality we have
\[
\int_0^T\int_{\mathbb{T}^d}|\zeta|\rho \dx \dt\leq \|\zeta\|_{L^r(\mathbb{T}^d)}\|\rho\|_{L^1(L^q(\dx),\dt)},
\]
which ends the proof.
\end{proof}

\begin{pro}
\label{droalfa}
Assume (R\ref{r1})-(R\ref{r4}). Let $w$ and $\rho$ solve \eqref{main2} and \eqref{rho} respectively.
Then for $0<\alpha<1$, and any $\delta_1>0$ there exists $C_{\delta_1}$ such that
\[
\int_0^T\int_{\mathbb{T}^d}|D(\rho^{\frac{\alpha}{2}})|^2dxdt\leq C_{\delta_1}+\delta_1 \int_0^T\int_{\mathbb{T}^d}|Dw|^2\rho(x,t)dxdt.
\]
\end{pro}

\begin{proof}
Multiplying the first equation in \eqref{rho} by $\rho^{\alpha-1}$ and integrating by parts, we obtain
\begin{align}
c_{\alpha}\int_0^T\int_{\mathbb{T}^d}|D(\rho^{\frac{\alpha}{2}})|^2dxdt=&
\frac{1}{\alpha}\int_{\mathbb{T}^d}(\rho^{\alpha}(x,T)-\rho^{\alpha}(x,0))dx
\\&+
(1-\alpha)\int_0^T\int_{\mathbb{T}^d}\rho^{\alpha-1}D_pF(x,Dw)\cdot D\rho dxdt
\\
\leq & C+\varepsilon\int_0^T\int_{\mathbb{T}^d}|D(\rho^{\frac{\alpha}{2}})|^2dxdt +
C_{\varepsilon}\int_0^T\int_{\mathbb{T}^d}|D_pF(x,Dw)|^2\rho^{\alpha}dxdt,
\label{droa}
\end{align}
for any $\varepsilon>0,$ where $c_{\alpha}=\frac{4(1-\alpha)}{\alpha^2}.$ Here we used
\[
\int_{\mathbb{T}^d}\rho^{\alpha}(x,0)dx,\quad\int_{\mathbb{T}^d}\rho^{\alpha}(x,T)dx\leq 1,
\] which is a consequence of \eqref{rhodensityadjsec} and Jensen's inequality.
Furthermore, using that  $\rho^{\alpha}\leq C_{\delta_1}+\delta_1\rho$ and (R\ref{r1}), the last term in the inequality \eqref{droa} can be bounded as follows
\[
\int_0^T\int_{\mathbb{T}^d}|D_pF(x,Dw)|^2\rho^{\alpha}dxdt\leq C_{\delta_1}+
\delta_1\int_0^T\int_{\mathbb{T}^d}|Dw|^2\rho dxdt.
\]
For $\varepsilon$ small enough  we get the result.
\end{proof}

\begin{rem}
In fact the expression $\rho(x,t)^{\alpha}$ does not always make sense since $\rho(x,0)=\delta_{x_0}.$ To fix this we consider the solution $\rho^{\varepsilon}$ to the equation \eqref{rho} but with initial value $\eta_{\varepsilon}$ instead of $\delta_{x_0},$ where $\eta_{\varepsilon}\colon\mathbb{T}^d\to\Rr$ are smooth compactly supported functions with $\int_{\mathbb{T}^d}\eta_{\varepsilon}(x)\dx=1$ and $\eta_{\varepsilon}\rightharpoonup\delta_{x_0}$. We carry out all the computations with $\rho^{\varepsilon}$ and then send $\varepsilon\to0$.
\end{rem}

%\begin{rem}
%In general $\rho(x,t)$ may vanish, so instead of using the negative powers of $\rho$ in the previous computation, we can carry out the same computations with the powers
%$\rho+\delta$ and then send $\delta$ to zero.
%\end{rem}

Combining the Proposition \ref{droalfa} and Corollary \ref{dw2ro} we conclude that
\begin{cor}Assume (R\ref{r1})-(R\ref{r4}). Let $w$ and $\rho$ solve  \eqref{main2} and \eqref{rho} respectively. Then for any $0<\alpha<1$, and any $\delta_1>0$ there exists $C_{\delta_1}$ such that
\label{droalf}
\[
\int_0^T\int_{\mathbb{T}^d}|D(\rho^{\frac{\alpha}{2}})|^2dxdt\leq C_{\delta_1}+C\delta_1\|\rho\|_{L^1(L^q(dx),dt)}+C\delta_1 \osc(w),
\] where $q$ is the conjugate exponent of $r$.
\end{cor}

Define
\begin{equation}
\label{ard}
\alpha_{rd}= 1+\frac 1 r -\frac 2 d.
%\max\left\{\frac{(d-2) r}{d (r-1)-r}}, 1+\frac 1 r -\frac 2 d\right\}
%\alpha_{rd}=\max\left\{\frac{r}{r-1}\frac{d-2}{d}, 1+\frac 1 r -\frac 2 d\right\}.
\end{equation}
Since by Assumption (R\ref{r2}) $r>d$, we have $\alpha_{rd}<1$.
%Further $\alpha_{rd}= 1+\frac 1 r -\frac 2 d$ for
%$r>d$.

\begin{pro}
\label{rol1lq}
Assume (R\ref{r1})-(R\ref{r4}). Let $w$ and $\rho$ solve  \eqref{main2} and \eqref{rho} respectively.
Then, for $\alpha>\alpha_{rd}$, there exists $0<\mu<1$ such that
\[
\|\rho\|_{L^1(L^q(dx),dt)}\leq C\left(\int_0^T\int_{\mathbb{T}^d}|D(\rho^{\frac{\alpha}{2}})|^2dxdt \right)^{\mu} +C,
\] where $q$ is the conjugate exponent of $r$.
\end{pro}
\begin{proof}
Recall that for any $1\leq p_0< p_1<\infty,\, 0<\theta<1$ we have the following interpolation inequality
\[
\|f\|_{L^{p_{\theta}}}\leq\|f\|_{L^{p_1}}^\theta\|f\|_{L^{p_0}}^{1-\theta
},
\]
where $p_{\theta}$ is given by
\[
\frac{1}{p_{\theta}}=\frac{\theta}{p_1}+\frac{1-\theta}{p_0}.
\]
If $d>2$, let $p=2^*,$ where $2^*$ is the Sobolev's conjugate exponent of $2$ given by $\frac{1}{d}=\frac{1}{2}-\frac{1}{2^*}$. If $d=2$
we take $p$ to be a sufficiently large exponent.
Take $p_0=1,$ $p_1=\frac{p\alpha}{2}$.
Let  $q$ be the conjugate exponent of $r$. Note that if $\alpha>\alpha_{rd}$ we have
\[
1<q< p_1.
\]
Then for $p_{\theta}=q$ we have
\[
\theta=\frac{1-\frac{1}{q}}{1-\frac{1}{p_1}}.
\]
By Sobolev's inequality
\[
\left(\int_{\mathbb{T}^d}\rho^{\frac{p\alpha}{2}}(x,t) \right)^{\frac{1}{p}}\leq C \left(\int_{\mathbb{T}^d}|D(\rho^{\frac{\alpha}{2}})(x,t)|^2dx \right)^{\frac{1}{2}} +C,
\]
and so
\[
\|\rho(.,t)\|_{L^{\frac{\alpha p}{2}}(\mathbb{T}^d)}\leq C \left(\int_{\mathbb{T}^d}|D(\rho^{\frac{\alpha}{2}})(x,t)|^2dx \right)^{\frac{1}{\alpha}}+C.
\]
Using $\|\rho(.,t)\|_{L^1}=1$ and the interpolation we get
\[
\|\rho(.,t)\|_{L^{q}(\mathbb{T}^d)}\leq C \left(\int_{\mathbb{T}^d}|D(\rho^{\frac{\alpha}{2}})(x,t)|^2dx \right)^{\mu}+C,
\]
with $\mu=\frac{\theta}{\alpha}$.
%Since $\theta<1$,
%$$\mu=  \frac{1-\frac{1}{q}}{\alpha-\frac{2}{2^*}}=\frac{1}{r(\alpha-\frac{2}{2^*})}=\frac{1}{r(\alpha-1+\frac{2}{d})},$$
For $\alpha>\alpha_{rd}$, we have $\mu<1$.
Then by Jensen's inequality
\[
\|\rho\|_{L^1(L^q(dx),dt)}=\int_0^T\|\rho(.,t)\|_{L^q(\mathbb{T}^d)}\leq C\left(\int_0^T\int_{\mathbb{T}^d}|D(\rho^{\frac{\alpha}{2}})|^2dxdt \right )^{\mu} +C,
\] where $q$ is the conjugate exponent of $r$.
\end{proof}
Combining Corollary \ref{droalf}  and Proposition \ref{rol1lq}, we get
\begin{cor}\label{drhoalestim}
Assume (R\ref{r1})-(R\ref{r4}). Let $w$ and $\rho$ solve  \eqref{main2} and \eqref{rho} respectively. Then, for  for $\alpha>\alpha_{rd}$ and any $\delta_1>0$ there exists $C_{\delta_1}$ such that
\[
\int_0^T\int_{\mathbb{T}^d}|D(\rho^{\frac{\alpha}{2}})|^2dxdt\leq C_{\delta_1}+C\delta_1 \osc(w).
\]
\end{cor}
Furthermore, using this with Proposition \ref{rol1lq} gives
\begin{cor}\label{l1lqrhoest}
Assume (R\ref{r1})-(R\ref{r4}). Let $w$ and $\rho$ solve \eqref{main2} and \eqref{rho} respectively. Then, for $\mu$ as in Proposition \ref{rol1lq}
\[
\|\rho\|_{L^1(L^q(dx),dt)}\leq C+C(\osc(w))^{\mu},
\] where $q$ is the conjugate exponent of $r$.
\end{cor}
Finally, from Corollaries \ref{dw2ro} and \ref{l1lqrhoest}, we infer
\begin{cor}\label{dwrhofinalest}
Assume (R\ref{r1})-(R\ref{r4}). Let $w$ and $\rho$ solve \eqref{main2} and \eqref{rho} respectively. Then
\[
\int_0^T\int_{\mathbb{T}^d}|Dw|^2\rho(x,t)dxdt\leq C+C \osc(w).
\]

\end{cor}

\begin{pro}\label{lipprop}
Assume (R\ref{r1})-(R\ref{r4}). Let $w$  solve \eqref{main2}.
Then $Lip(w)\leq C.$
\end{pro}

\begin{proof}
 Let $\eta=D_{x_i}w$, then it satisfies the equation
 \[
 \eta_t+D_pF(x,Dw) D\eta+\Delta \eta=-\widehat{F}_{x_i}(x,Dw)+D_{x_i}(\zeta).
 \]
 Take $\phi(t)$ to be smooth with $\phi(0)=1$ and $\phi(T)=0$. Let $v=\phi \eta$,
 then it satisfies
 \[
 v_t+D_pF(x,Dw)\cdot D_xv+\Delta v=-\phi\widehat{F}_{x_i}(x,Dw)+\phi D_{x_i}(\zeta)+\phi'D_{x_i}w.
 \]
 Integrating with respect to $\rho$

 \[
 -v(x_0,0)=\int_0^T\int_{\mathbb{T}^d}-\phi \widehat{F}_{x_i}\rho+\phi D_{x_i}(\zeta)\rho+\phi'D_{x_i}w\rho dxdt.
 \]
Using that $|D_{x_i}w\rho|\leq\varepsilon |Dw|^2\rho+C_{\varepsilon}\rho$ for small $\varepsilon>0$ \begin{equation}
 \label{lipu}
 |v(x_0,0)|\leq\int_0^T\int_{\mathbb{T}^d}C|\widehat{F}_{x_i}|\rho+C\rho  +C\varepsilon |Dw|^2\rho \dx\dt+\int_0^T\left|\int_{\mathbb{T}^d}D_{x_i}(\zeta)\rho \dx\right
 |\dt.
 \end{equation}
 The first term in the right-hand side of \eqref{lipu} can be estimated using (R\ref{r3}) and Corollary \ref{l1lqrhoest}:
  \begin{equation}
 \label{lipufhutbdd}\begin{split}
\int_0^T\int_{\mathbb{T}^d}|\widehat{F}_{x_i}|\rho \dx&\leq \int_0^T\int_{\mathbb{T}^d}C\rho+\psi|Dw|^\beta\rho\dx
\leq \int_0^T\int_{\mathbb{T}^d} C\rho+C_\varepsilon \psi ^\frac{2}{2-\beta}\rho+\varepsilon|Dw|^2\rho\dx\\&
\leq C+C_\varepsilon\|\rho\|_{L^1(L^q(dx),dt)} \|\psi\|_{L^\frac{2r}{2-\beta}}^\frac{2}{2-\beta}+\varepsilon\int_0^T\int_{\mathbb{T}^d}|Dw|^2\rho\dx\\&
\leq C+C(\osc(w))^{\mu}+\varepsilon\int_0^T\int_{\mathbb{T}^d}|Dw|^2\rho\dx.
\end{split}
\end{equation}
Let us now estimate the last term in the right-hand side of \eqref{lipu}. We have
 \[
 \int_{T^d}D_{x_i}(\zeta)\rho dx=-\int_{\mathbb{T}^d}\zeta D_{x_i}(\rho) \dx=\frac{2}{\alpha}\int_{\mathbb{T}^d}\zeta\rho^{1-\alpha/2}D_{x_i}(\rho^{\alpha/2}) \dx.
 \]
Thus
\[
 \int_0^T|\int_{\mathbb{T}^d}D_{x_i}(\zeta)\rho dx|dt\leq C \int_0^T\int_{\mathbb{T}^d}\zeta^2\rho^{2-\alpha} dxdt+C\int_0^T\int_{\mathbb{T}^d}|D(\rho^{\alpha/2})|^2\dx\dt.
\]
We estimate the first term of the previous inequality as follows
\[
\int_{\mathbb{T}^d}\zeta^2\rho^{2-\alpha} dx\leq\|\zeta^2\|_{L^{\frac{r}{2}}(\mathbb{T}^d)}\|\rho^{2-\alpha}\|_{L^{\frac{r}{r-2}}(\mathbb{T}^d)}= \|\zeta\|_{L^r(\mathbb{T}^d)}^2\|\rho\|^{2-\alpha}_{L^{{\frac{(2-\alpha)r}{r-2}}}(\mathbb{T}^d)}.
\]
%For $d<r$ and $\alpha$ close to $1$ we have $\frac{(2-\alpha)r}{r-2}\leq\frac{2^*\alpha}{2}=\frac{\alpha d}{d-2}$ (in dimension $d=2$ this holds for $p$ large %enough in place of $2^*$).
Then by Sobolev inequality and \eqref{rhodensityadjsec}
\begin{equation}
\|\rho\|_{L^{\frac{2^*\alpha}{2}}(\mathbb{T}^d)}
\leq C\|D(\rho^{\alpha/2})\|_{L^2(\mathbb{T}^d)}^{\frac{2}{\alpha}}
+C.
\end{equation}
If $d>2$, choose now $\alpha>\alpha_{rd}$ so that
$\frac{(2-\alpha)r}{r-2}<\frac{2^*\alpha}{2}=\frac{\alpha d}{d-2}$.
In dimension $2$ replace in the previous condition $2^*$
by a sufficiently large $p$. Note that such choice is possible since for $\alpha=1$ we have
$\frac r {r-2}<\frac d {d-2}$.
Using interpolation we get
\begin{equation}
\|\rho\|_{L^{\frac{(2-\alpha)r}{r-2}}(\mathbb{T}^d)}\leq\|\rho\|_{L^1(\mathbb{T}^d)}^{1-\theta_1}
\|\rho\|^{\theta_1}_{L^{\frac{2^*\alpha}{2}}(\mathbb{T}^d)}\leq C\|D(\rho^{\alpha/2})\|^{\frac{2\theta_1}{\alpha}}_{L^2(\mathbb{T}^d)}
+C,
\end{equation}
where $\theta_1$ is defined by $\frac{r-2}{(2-\alpha)r}=\frac{1-\theta_1}{1}+\frac{2\theta_1}{2^*\alpha}$.
As $\alpha \to 1$ we have $\theta_1\to \frac{d}{r}$. Then if
 $\alpha>\alpha_{rd}$ sufficiently close to $1$ we have
\[
\frac{(2-\alpha)\theta_1}{\alpha}<1.
\]
Then, using Jensen's inequality we get
\begin{equation}
\label{dzr}
\begin{split}
\int_0^T|\int_{\mathbb{T}^d}D(\zeta)\rho dx|dt&\leq C\int_0^T\left (\int_{\mathbb{T}^d}|D(\rho^{\alpha/2})|^2 \right )^{\frac{(2-\alpha)\theta_1}{\alpha}}+C\delta_1 \osc(w)+C
\\
&\leq C\left (\int_0^T\int_{\mathbb{T}^d}|D(\rho^{\frac{\alpha}{2}})|^2dxdt\right )^{\frac{(2-\alpha)\theta_1}{\alpha}}+C\delta_1 \osc(w)+C.
\end{split}
\end{equation}
Note that we can choose $x_0$ and $i$ such that
\[
Lip(w)\leq d |v(x_0, 0)|= d|D_{x_i} w(x_0)|.
\]
Then
combining  the inequalities \eqref{lipu}, \eqref{lipufhutbdd} and \eqref{dzr}, Corollaries \ref{drhoalestim} and \ref{dwrhofinalest},   and using $\osc(w)\leq C Lip(w)$, we obtain
$$
Lip(w)=\sup_{x_0\in\mathbb{T}^d}d |v(x_0)|\leq C+C(\varepsilon+\delta_1) Lip(w)+(C+C Lip(w))^{\frac{2-\alpha}{\alpha}\theta_1}
$$
choosing $\varepsilon,\delta_1$ small and since $\frac{2-\alpha}{\alpha}\theta_1<1$, for $\alpha$ close enough to 1, we obtain the result.
\end{proof}

\begin{cor}
\label{ulip}
Assume (A\ref{qzv})-(A\ref{bdv2}). Let $(u,m,V,\overline{H})$ solve the system \eqref{main}. Assume further the a-priori bounds  $\|g(m)\|_{L^r}\leq C$ for $r>d$. Then there exists a constant $C>0$ such that $\|u\|_ {W^{2,r}(\mathbb{T}^d)},\, \|u\|_{W^{1,{\infty}}(\mathbb{T}^d)}\leq C.$
\end{cor}

\begin{proof}
The property $\|u\|_{W^{1,{\infty}}(\mathbb{T}^d)}\leq C$ follows from Proposition \ref{lipprop}, estimate \eqref{H_0dxbound} and the fact that
\[
F(x,p)=H(x,p,m(x),V(x))-\overline{H}
\]
satisfies the hypothesis (R\ref{r1})-(R\ref{r3})  with $\zeta(x)=g(m(x))+C$, as we show now.

Let us check (R\ref{r1}).
Using Assumptions (A\ref{bbl1}), (A\ref{qbd})  and Proposition \ref{bV}, we get
\begin{equation*}\begin{split} H(x,p,m,V)&=-L(x,p,m,V)+D_pH(x,p,m,V)\cdot p\\&
\leq -cH_0(x,p,m,V)-g(m)+C+\frac{c}{C}|D_pH(x,p,m,V)|^2+C|p|^2\\&
\leq -cH_0(x,p,m,V)-g(m)+cH_0(x,p,m,V)+C|p|^2+C\\&
=C|p|^2-g(m)+C.
\end{split}\end{equation*}
This and  (A\ref{qzv}) imply that
$$H_0(x,p,m,V)\leq C|p|^2+C,$$ and then, by (A\ref{qbd}) that
$$|D_pH(x,p,m,V)|^2\leq C|p|^2+C,$$ i.e.,
$F(x,p)=H(x,p,m(x),V(x))-\overline{H}$ satisfies (R\ref{r1}).
The property (R\ref{r2}) is a consequence of  (A\ref{bbl1}) and Proposition \ref{bV}. Assumption (A\ref{bdv2}) and estimate \eqref{H_0dxbound} imply (R\ref{r3}).

Once we know that $\|u\|_{W^{1,{\infty}}(\mathbb{T}^d)}\leq C$, from the first equation of \eqref{main} and $|\overline{H}|\leq C$, we infer that
\[
|\Delta u|\leq |g(m)|+C.
\]
Since by assumption  $\|g(m)\|_{L^r}\leq C$, from the elliptic theory we get  $\|u\|_ {W^{2,r}(\mathbb{T}^d)}\leq C$.
\end{proof}

The next Corollary generalizes the result in Corollary \ref{ulpsc}.

\begin{cor}
\label{u3reg}
Assume in addition to the hypothesis of  Corollary \ref{ulip}, that $\|D(g(m))\|_{L^2}\leq C.$  Then there exists a constant $C>0$ such that $\|u\|_ {W^{3,2}(\mathbb{T}^d)}\leq C.$
\end{cor}

\begin{proof}
We have
$$
D\Delta u=-D(H(x,Du,m,V))=-\widehat{H}_x+D(g(m))-D^2uD_pH
$$
which combined with the Corollary \ref{ulip} and Assumption (A\ref{bdv2}) gives $\|D\Delta u\|_{L^2(\mathbb{T}^d)}\leq C$, hence $\|u\|_ {W^{3,2}(\mathbb{T}^d)}\leq C.$
\end{proof}
Combining this with the Corollary \ref{gmint} we get:
\begin{cor}
\label{lnml}
Assume (A\ref{qzv})-(A\ref{bdv2}). Let $(u,m,V,\overline{H})$ solve the system \eqref{main}. Furthermore, suppose that  one of the following assumptions is satisfied:
\begin{itemize}
\item[(i)] (A\ref{g}\ref{lnm}
\item[(ii)]  (A\ref{g}\ref{mgama}, with any $\gamma>0$ if $d\leq 4,$ and $\gamma<\frac{1}{d-4}$ if $d\geq 5.$

\end{itemize}
Then there exists a constant $C>0$ such that $\|u\|_ {W^{1,{\infty}}(\mathbb{T}^d)},\,\|u\|_ {W^{2,r}(\mathbb{T}^d)},\,\|u\|_ {W^{3,2}(\mathbb{T}^d)}\leq C.$
\end{cor}

\begin{cor}
\label{mlinfty}
Let $(u,m,V,\overline{H})$ solve the system \eqref{main}. Assume either
\begin{itemize}
\item[A.]
 (A\ref{qzv})-(A\ref{bdv2}) and  one of the following assumptions is satisfied:
\begin{itemize}
\item[(i)] (A\ref{g}\ref{lnm}
\item[(ii)]  (A\ref{g}\ref{mgama}, with any $\gamma>0$ if $d\leq 4,$ and $\gamma<\frac{1}{d-4}$ if $d\geq 5.$

\end{itemize}
\end{itemize}
or
\begin{itemize}
\item[B.]
 (H1),(H2) holds  with $\epsilon\in[0,\epsilon_0]$,  (A\ref{g}\ref{lnm} and $d\leq 3$.
\end{itemize}
Then there exists a constant $C>0$ such that
\[\|\ln m\|_{W^{1,\infty}(\mathbb{T}^d)} \leq C.\]  In particular there exists a uniform constant $\bar{m}>0$ such that $m\geq \bar{m}.$  Furthermore, for any $q> 1$  there exists a constant $C_q>0$ such that
$$\|u\|_{W^{2,q}(\mathbb{T}^d)},\,\|m\|_ {W^{2,q}(\mathbb{T}^d)}\leq C_q.$$
\end{cor}
\begin{proof}
%Let us first prove the estimates  $\|m\|_{L^{\infty}(\mathbb{T}^d)},\,\left\|\frac{1}{m}\right\|_{L^{\infty}(\mathbb{T}^d)} \leq C$.
Take any $r\in\Rr$ multiply the second equation of \eqref{main} by $m^r$, $r\neq 0$, or by $\ln m$ for $r=0$ and integrate by parts:
\begin{equation*} \int_{\mathbb{T}^d} m^{r-1}|Dm|^2-m^rD_pH\cdot Dmdx=0, \forall r\in\Rr.\end{equation*}
Then using Corollaries \ref{lnml} or \ref{ulpsc} and H\"{o}lder's inequality:
\begin{equation*} \int_{\mathbb{T}^d} m^{r-1}|Dm|^2dx\leq C\int_{\mathbb{T}^d}m^r|Dm|\dx\leq C \left(\int_{\mathbb{T}^d} m^{r-1}|Dm|^2dx\right)^{\frac{1}{2}}\left(\int_{\mathbb{T}^d} m^{r+1}dx\right)^{\frac{1}{2}},\end{equation*}
thus
\begin{equation}\label{dm}c_r\int_{\mathbb{T}^d} |Dm^{\frac{r+1}{2}}|^2dx= \int_{\mathbb{T}^d} m^{r-1}|Dm|^2dx\leq C\int_{\mathbb{T}^d} m^{r+1}dx.\end{equation}
Note that $c_rm^{r-1}|Dm|^2=|Dm^{\frac{r+1}{2}}|^2$ with $c_r=\frac{(r+1)^2}{4}$. By Sobolev's Theorem, if $m^{\frac{r+1}{2}}\in H^1(\mathbb{T}^d)$ then $m^{\frac{r+1}{2}}\in L^{2^*}(\mathbb{T}^d)$ and
\begin{equation*}
\left(\int_{\mathbb{T}^d} m^{\frac{2^*}{2}(r+1)}\right)^\frac{2}{2^*}\leq C\int_{\mathbb{T}^d} c_r m^{r-1}|Dm|^2+m^{r+1}dx,\end{equation*} where $2^*=\frac{2d}{d-2}.$ Then we have
\begin{equation*}
 \left(\int_{\mathbb{T}^d} m^{\frac{2^*}{2}(r+1)}\right)^\frac{2}{2^*}\leq C(r^2+1)\int_{\mathbb{T}^d} m^{r+1}dx.
 \end{equation*}
 Thus for any $r>0$
 \begin{equation}\label{mlinftyinteg}
 \left(\int_{\mathbb{T}^d} m^{\beta r}\right)^\frac{1}{\beta}\leq C(r^2+1)\int_{\mathbb{T}^d} m^{r}dx,
  \end{equation} and
   \begin{equation}\label{1divmlinftyinteg}
 \left(\int_{\mathbb{T}^d} \frac{1}{m^{\beta r}}\right)^\frac{1}{\beta}\leq C(r^2+1)\int_{\mathbb{T}^d} \frac{1}{m^{r}}dx.
 \end{equation}
 where $\beta=\frac{2^*}{2}>1.$
 Since $\int_{\mathbb{T}^d} m\dx=1$,  arguing as in the last part of the proof of Theorem \ref{1/mintegrab}, from \eqref{mlinftyinteg} we get $\|m\|_{L^{\infty}(\mathbb{T}^d)}\leq C.$

Next, if  (A\ref{g}\ref{lnm} holds, then by Proposition \ref{1/mintegrabgeneralprop} we know that there exists $r_0>0$ such that $\int_{\mathbb{T}^d} \frac{1}{m^{r_0}}dx\leq C$.
 Hence, again arguing as in Theorem \ref{1/mintegrab}, from  \eqref{1divmlinftyinteg} we get $\left\|\frac{1}{m}\right\|_{L^{\infty}(\mathbb{T}^d)} \leq C$.

In both cases A and B the Lipschitz estimates on $u$ from Corollaries \ref{lnml} and \ref{ulpsc}, and the estimates just proven imply that
 $\|\Delta u\|_{L^{\infty}(\mathbb{T}^d)}\leq C$. In particular
\begin{equation*}\|u\|_{W^{2,q}(\mathbb{T}^d)}\leq C_q\quad\text{for any }q>1.\end{equation*}

Now, let us show that $\|\log m\|_{W^{1,\infty}(\mathbb{T}^d)} \leq C$. The function $v=\log m$ is solution of
$$\Delta v +|D v|^2-b(x)\cdot Dv-\zeta(x)=0,$$ where $b(x)=D_pH(x,Du,m,V)$ and $\zeta=\div(b).$ By Proposition \ref{lminh1}, we know that $\|Dv\|_{L^2(\mathbb{T}^d)}\leq C.$ Moreover, the $W^{2,q}$ estimates  on $u$ and (A\ref{bdv1}) imply that $D_x(b),\,\zeta\in L^q(\mathbb{T}^d)$ for any $q>1$.
Hence, the Hamiltonian $F(x,p)=|p|^2-b(x)\cdot p-\zeta(x)$ satisfies Assumptions (R\ref{r1})-(R\ref{r3}). Proposition \ref{lipprop} then gives $Lip(v)\leq C.$
In particular, from  $\|\log m\|_{L^{\infty}(\mathbb{T}^d)} \leq C$ we infer that there exists $\bar{m}>0$ such that $m\geq \bar{m}$. Moreover, the estimates
$\left\|\frac{Dm}{m}\right\|_{L^{\infty}(\mathbb{T}^d)},  \| m\|_{L^{\infty}(\mathbb{T}^d)} \leq C$ imply $\| Dm\|_{L^{\infty}(\mathbb{T}^d)} \leq C$.

Finally, from  the equation for $m$
\[
\Delta m= bDm+\div(b)m,
\] and the estimates just proven we get $\|m\|_{W^{2,q}(\mathbb{T}^d)} \leq C_q$ for any $q> 1$.
\end{proof}

%{\bf the next theorem seems to require $d\leq 3$}

%\begin{cor}
%Under the assumptions of the Corollary \ref{ulip}, there exists a constant $C>0$ such that
%$\|m\|_ {W^{2,2}}\leq C.$
%\end{cor}
%\begin{proof}
%We have that $\Delta m=\div(D_pH)m+D_pHDm$ is bounded in $L^2$, thus $m$ is bounded in $W^{2,2}.$
%\end{proof}

\begin{teo}
\label{usmooth}
Let $(u,m,V,\overline{H})$ solve the system \eqref{main}. Assume either
\begin{itemize}
\item[A.]
 (A\ref{qzv})-(A\ref{bdv2}) and  one of the following assumptions is satisfied:
\begin{itemize}
\item[(i)] (A\ref{g}\ref{lnm}
\item[(ii)]  (A\ref{g}\ref{mgama}, with any $\gamma>0$ if $d\leq 4,$ and $\gamma<\frac{1}{d-4}$ if $d\geq 5.$
\end{itemize}
\end{itemize}
or
\begin{itemize}
\item[B.]
 (H1),(H2) holds  with $\epsilon\in[0,\epsilon_0]$,  (A\ref{g}\ref{lnm} and $d\leq 3$.
\end{itemize}
Then there exist constants $C_{k,q}$ such that $\|u\|_{W^{k,q}(\mathbb{T}^d)}, \|m\|_{W^{k,q}(\mathbb{T}^d)},\|V\|_{W^{k,q}(\mathbb{T}^d,\Rr^d)}\leq C_{k,q}$ for any $q,k\geq1$.
\end{teo}

\begin{proof}
Corollary \ref{mlinfty}  gives \ $\|m \|_{W^{1,\infty}}\leq C$ and $\|u\|_{ W^{2,q}}\leq C_q$ for every $1<q<\infty$. Differentiating the first equation in \eqref{main} yields
\[
D\Delta u=-D(H(x,Du,m,V))=-\widehat{H}_x+D(g(m))-D^2uD_pH\text{ is bounded in }L^q
\]
thus $\|u\|_{ W^{3,q}}\leq C_{3,q}$ for all $1<q<\infty.$

Therefore, from assumption (A\ref{difp}) and the second and third equations of \eqref{main}, we get
$\|m\|_{ W^{2,q}},\,\|V\|_{ W^{2,q}}\leq C_{2,q}$ for all $1<q<\infty.$ A bootstrap argument completes the proof of the theorem.
%Then since $f=\ln m$ satisfies
%\[
%\Delta f+|D f|^2-\div(D_pH)-D_pHDf=0,
%\]
%we have $W^{2,q}$ bounds of $f$ for all $q>1$. We prove by induction that for every positive integer $k$, we can bound  $W^{k+1,q}$ norms of $u$ and $W^
%{k,q}$ norms of  $f$ for all $q>1,$ this will yield the result.
%Assume we managed to get bounds of $u$ in $W^{k+1,q}$ and $f$ in $W^{k,q}$ for all $1<q<\infty.$
%Let $A(x,p)=H(x,p,m,V)+g(m(x))$, from the Assumption \ref{difp} and the previous bounds on $m,V$ we have that $A$ is smooth in $x,p$ with uniformly %bounded derivatives, hence one can bound $W^{k,q}$ norms of $A( x,Du(x))$ for all $q>1$, provided  there are bounds on $ W^{k+1,q}$ norms of $u$. Note %also that the bounds on $ W^{k,q}$ norms of $f$ for all $q>1,$ imply bounds on $W^{k,q}$ norms of $g(m)$ for all $q>1$. We have $\Delta u=A(x,Du(x))-g(m)+%\Hh$ thus the induction hypothesis imply we can bound the $W^{k+2,q}$ norms of $u$ , then from
%$
%\Delta f=-|D f|^2+\div(D_pH)+D_pHDf
%$
%we obtain bounds for
%$f$ in $W^{k+1,q}$, this end the induction step. Finally from third equation of \eqref{main} $V(x)=D_pA(x,Du(x),m,V)$
%since $A$ has uniformly bounded derivatives yields the required bounds on $V$.
\end{proof}

%%%%%%%%%%%%%%%%%%%%%%%%%%%%%%%%%%%%%%%%%%%%
%%%%%%%%%%%%%%%%%%%%%%%%%%%%%%%%%%%%%%%%%%%%

\section{Existence by continuation method}
\label{exist}

To prove the existence of smooth solutions to \eqref{main} let us write it in an equivalent form
\begin{equation*}
\begin{cases}
\Delta m-\div(D_pH(x, Du,m,V)m)=0\\
\Delta u+H(x, Du,m,V)=\overline{H}\\
V=D_pH(x, Du,m,V),
\end{cases}
\end{equation*}
and consider a parameterized family of Hamiltonians:
 \[
 H_{\lambda}(x,p,m,V)=\lambda H(x,p,m,V)+(1-\lambda)\left(\frac{|p|^2}{2}-g(m)\right),\quad 0\leq \lambda\leq 1,
 \]
with the corresponding system of PDE's:
 \begin{equation}
\label{mainl}
\begin{cases}
\Delta m_{\lambda}-\div(D_pH_{\lambda}(x,Du_{\lambda},m_{\lambda},V_{\lambda})m_{\lambda})=0\\
\Delta u_{\lambda}+H_{\lambda}(x,Du_{\lambda},m_{\lambda},V_{\lambda})=\overline{H}_{\lambda}\\
V_{\lambda}=D_pH_{\lambda}(x,Du_{\lambda},m_{\lambda},V_{\lambda})\\
\int_{\mathbb{T}^d}u_{\lambda}\dx=0\\
\int_{\mathbb{T}^d}m_{\lambda}\dx=1.
\end{cases}
\end{equation}
First let us start with some notation and hypothesis. Let
\[
\dot{H}^k(\mathbb{T}^d,\Rr)=\{\,f\in H^k(\mathbb{T}^d,\Rr)|\int_{\mathbb{T}^d}f\dx=0\,\}.
\]
Consider the Hilbert space $F^k=\dot{H}^k(\mathbb{T}^d,\Rr)\times H^k(\mathbb{T}^d,\Rr)\times L^2(\mathbb{T}^d,\Rr^d)\times\Rr$ with the norm
\[
\|w\|^2_{F^k}=\|\psi\|^2_{\dot{H}^k(\mathbb{T}^d,\Rr)}+\|f\|^2_{H^k(\mathbb{T}^d,\Rr)}+\|W\|^2_{L^2(\mathbb{T}^d,\Rr^d)}+|h|^2,
\]
for $w=(\psi,f,W,h)\in F^k.$ We assume that $H$ can be extended from the space $\chi(\mathbb{T}^d)$ to the space $L^2(\mathbb{T}^d,\Rr^d).$ Note that by Sobolev's embedding theorem, $H$ is well defined on the set of positive functions $m\in H^k(\mathbb{T}^d,\Rr)$ with big enough $k$. We denote this set by $H^k_+(\mathbb{T}^d,\Rr)$, it is well defined for large $k$s and is an open subset in $H^k(\mathbb{T}^d,\Rr).$\\
For a point $I=(x,p,m,V)\in\mathbb{T}^d\times\Rr^d \times H^k_+(\mathbb{T}^d,\Rr)\times L^2(\mathbb{T}^d,\Rr^d)$ we define $\mathcal{A}^0_{\lambda,I}\colon\Rr^d\to \Rr,\mathcal{B}^0_{\lambda,I}\colon\Rr^d\to \Rr^d$ by
\[
\mathcal{A}^0_{\lambda,I}(w)=D_pH_{\lambda}(x,p,m,V)w,\quad
\mathcal{B}^0_{\lambda,I}(w)=D^2_{pp}H_{\lambda}(x,p,m,V)w,
\]
$\mathcal{A}^1_{\lambda,I}\colon H^k(\mathbb{T}^d,\Rr)\to \Rr$, $\mathcal{B}^1_{\lambda,I}\colon H^k(\mathbb{T}^d,\Rr)\to \Rr^d$ by

\[
\mathcal{A}^1_{\lambda,I}(f)=D_mH_{\lambda}(x,p,m,V)(f)+g'(m(x))f(x),\quad
\mathcal{B}^1_{\lambda,I}(f)=D^2_{pm}H_{\lambda}(x,p,m,V)(f),
\]
and $\mathcal{A}_{\lambda,I}^2\colon L^2(\mathbb{T}^d,\Rr^d)\to \Rr$, $\mathcal{B}_{\lambda,I}^2\colon L^2(\mathbb{T}^d,\Rr^d)\to \Rr^d$ by
\[
\mathcal{A}^2_{\lambda,I}(W)=D_{V}H_{\lambda}(x,p,m,V)(W),\quad
\mathcal{B}^2_{\lambda,I}(W)=D^2_{pV}H_{\lambda}(x,p,m,V)(W).
\]
In principle $\mathcal{A}^1_{\lambda,I}(f)$ is only defined for a smooth $f$, but we are implicitly assuming that the term $g'(m(x))f(x)$ cancels a corresponding term in $D_mH_{\lambda}(p,x,m,V)(f)$, as will be required in hypothesis $B$\ref{HvHmreg}.
\renewcommand{\labelenumi}{({\bf B\arabic{enumi}})}
\begin{enumerate}
\item
\label{Hcont}
We assume that for $(p,m,V)\in\Rr^d \times H^k_+(\mathbb{T}^d,\Rr)\times L^2(\mathbb{T}^d,\Rr^d)$ we have $H(x,p,m,V)\in H^k(\mathbb{T}^d,\Rr)$ for every $k$ big enough. Note that for big $k$ ( $k>d$), $m\in H^k_+(\mathbb{T}^d,\Rr)$ implies $g(m)\in H^k(\mathbb{T}^d,\Rr).$

We further assume that $H(x,p,m,V), D^2_{pp}H(x,p,m,V), D^2_{px}H(x,p,m,V)$ are continuous in $m$ with respect to the uniform convergence, and in $V$ with respect to the convergence in $L^2(\mathbb{T}^d,\Rr^d).$ We assume $H(x,p,m,V)$ and $D_{p}H(x,p,m,V)$ have Fr\'echet derivatives in $V$, thus the operators $\mathcal{A}^2_{\lambda,I}(W),
\mathcal{B}^2_{\lambda,I}(W)$ are well defined.

\item
  \label{HvHmreg}
%\begin{hp}
We assume that for any $f\in H^k(\mathbb{T}^d,\Rr)$ and $W\in L^2(\mathbb{T}^d,\Rr^d)$ the functions $\mathcal{A}^1_{\lambda,I}(f),$ $\mathcal{B}^1_{\lambda,I}(f),$ $\mathcal{A}^2_{\lambda,I}(W),$ $\mathcal{B}^2_{\lambda,I}(W)$ are smooth in $x$ and $p$.
\item
\label{HvHmdb}
For any positive integer $l$, any point $(m,V)\in H^k_+(\mathbb{T}^d,\Rr)\times L^2(\mathbb{T}^d,\Rr^d)$ and any number $R>0$, there exists a constant $C(l,m,V,R)$ such that
\[
|D^l_{x,p}\mathcal{A}^1_{\lambda}(f)|, |D^l_{x,p}\mathcal{B}^1_{\lambda}(f)|\leq C(l
,m,V,R)\|f\|_{L^2}\quad \forall f\in L^2(\mathbb{T}^d,\Rr), \text{ for all } x\in\mathbb{T}^d, |p|\leq R,
\]
and
\[
|D^l_{x,p}\mathcal{A}^2_{\lambda}(W)|, |D^l_{x,p}\mathcal{B}^2_{\lambda}(W)|\leq C(l,m,V,R)\|W\|_{L^2}\quad \forall W\in L^2(\mathbb{T}^d,\Rr^d),\text{ for all } x\in\mathbb{T}^d,|p|\leq R.
\]
%\end{hp}
\suspend{enumerate}
Because of the structure of $H_{\lambda}$ it suffices to check that both (B\ref{HvHmreg}) and (B\ref{HvHmdb}) hold when $\lambda=1$.

Thus for a point $(P,m,V)\in C^{\infty}(\mathbb{T}^d,\Rr^d)\times C^{\infty}(\mathbb{T}^d,\Rr)\times L^2(\mathbb{T}^d,\Rr^d),$ we can define the operators
$\mathcal{A}^1_{\lambda},\mathcal{B}^1_{\lambda}\colon H^k(\mathbb{T}^d,\Rr)\to C^{\infty}(\mathbb{T}^d,\Rr)$
by
\[
\mathcal{A}^1_{\lambda}(f)(x)=\mathcal{A}^1_{\lambda,I(x)}(f),\text{ and }\mathcal{B}^1_{\lambda}(f)(x)=\mathcal{B}^1_{\lambda,I(x)}(f),
\]
where $I(x)=(x,P(x),m,V)$.

Similarly, we define the operators $\mathcal{A}^2_{\lambda},\mathcal{B}^2_{\lambda}\colon L^2(\mathbb{T}^d,\Rr^d)\to C^{\infty}(\mathbb{T}^d,\Rr^d)$
by
\[
\mathcal{A}^2_{\lambda}(W)(x)=\mathcal{A}^2_{\lambda,I(x)}(W),\text{ and }\mathcal{B}^2_{\lambda}(W)(x)=\mathcal{B}^2_{\lambda,I(x)}(W).
\]
\resume{enumerate}
\item\label{invHpv}
We assume further that the linear mapping $Id-\mathcal{B}^2_{\lambda}\colon L^2(\mathbb{T}^d,\Rr^d)\to L^2(\mathbb{T}^d,\Rr^d)$ is invertible. Since $\mathcal{B}^2_{\lambda}=\lambda\mathcal{B}^2_1$, it is sufficient for the invertibility of $Id-\mathcal{B}^2_{\lambda}$ to have $\|\mathcal{B}^2_1\|_{L^2\to L^2}<1$.
\suspend{enumerate}
We consider now the linearization of \eqref{main} at the point $(\lambda_0,I_{\lambda_0})$
in the direction $(\psi,f,W,\bar h)$
 \begin{equation}
\label{linr}
\begin{cases}
\Delta f-\div\left(V_{\lambda_0}f\right)-\div\left[\left(\mathcal{B}^0_{\lambda_0,I_{\lambda_0}}(D\psi)+
\mathcal{B}^1_{\lambda_0,I_{\lambda_0}}(f)+
\mathcal{B}^2_{\lambda_0,I_{\lambda_0}}(W)\right)m_{\lambda_0}\right]=0\\
\Delta \psi+\mathcal{A}^0_{\lambda_0,I_{\lambda_0}}(D\psi)+\mathcal{A}^1_{\lambda_0,I_{\lambda_0}}(f)-g'(m_{\lambda_0})f+
\mathcal{A}^2_{\lambda_0,I_{\lambda_0}}(W)-\overline{h}=0\\
W=\mathcal{B}^0_{\lambda_0,I_{\lambda_0}}(D\psi)+
\mathcal{B}^1_{\lambda_0,I_{\lambda_0}}(f)+
\mathcal{B}^2_{\lambda_0,I_{\lambda_0}}(W),\\
\int_{\mathbb{T}^d}\psi\dx=0\\
\int_{\mathbb{T}^d}f\dx=0.
\end{cases}
\end{equation}
Where $I_{\lambda_0}(x)=(x,Du_{\lambda_0}(x),m_{\lambda_0},V_{\lambda_0}).$
Multiplying the second equation by $f$ and subtracting the first equation multiplied by $\psi$ and integrating by parts we get:
\begin{equation}
\begin{split}
0=\int_{\mathbb{T}^d}[&f\mathcal{A}^1_{\lambda_0,I_{\lambda_0}(x)}(f)-g'(m_{\lambda_0})f^2+
f\mathcal{A}^2_{\lambda_0,I_{\lambda_0}(x)}(W)-m_{\lambda_0}D\psi\mathcal{B}^0_{\lambda_0,I_{\lambda_0}(x)}(D\psi(x))\\&-
m_{\lambda_0}D\psi\mathcal{B}^1_{\lambda_0,I_{\lambda_0}(x)}(f)-
m_{\lambda_0}D\psi\mathcal{B}^2_{\lambda_0,I_{\lambda_0}(x)}(W)]\dx,
\end{split}
\end{equation}
where we used $\mathcal{A}^0_{\lambda_0,I_{\lambda_0}(x)}(D\psi)=V_{\lambda_0}D\psi(x).$

For a point $I(x)=(x,P(x),m,V)$ where $(P,m,V)\in H^k(\mathbb{T}^d,\Rr^d)\times H^k_+(\mathbb{T}^d,\Rr)\times L^2(\mathbb{T}^d,\Rr^d)$
we define $\mathcal{H}_{\lambda,I}\colon H^k(\mathbb{T}^d,\Rr^d)\times H^k(\mathbb{T}^d,\Rr)\times L^2(\mathbb{T}^d,\Rr^d)\to\Rr$ by
\begin{equation}
\begin{split}
\mathcal{H}_{\lambda,I}(Q,f,W)=\int_{\mathbb{T}^d}[&mQ\mathcal{B}^0_{\lambda,I(x)}(Q)+
mQ\mathcal{B}^1_{\lambda,I(x)}(f)+g'(m)f^2\\&-f\mathcal{A}^1_{\lambda,I(x)}(f)+
mQ\mathcal{B}^2_{\lambda,I(x)}(W)-
f\mathcal{A}^2_{\lambda,I(x)}(W)]\dx.
\end{split}
\end{equation}
Note that $\mathcal{H}_0(Q ,f,W)=\int_{\mathbb{T}^d}m|Q|^2+g'(m)|f(x)|^2\dx$  and
$
\mathcal{H}_{\lambda}=(1-\lambda)\mathcal{H}_{0}+\lambda\mathcal{H}_{1}.
$

\resume{enumerate}
\item\label{monotcond}
We suppose that there exists a constant $C$ such that for any $I(x)=(x,Du(x),m,V)$, where $(u,m,V,\overline{H})$ is a solution
to \eqref{mainl}, and for all $\lambda\in[0,1]$:
\[
\mathcal{H}_{\lambda, I}(Q,f,W)\geq \theta\int_{\mathbb{T}^d}m|Q|^2+|f(x)|^2 -C(W-\mathcal{B}^0_{\lambda,I(x)}(Q)-
\mathcal{B}^1_{\lambda,I(x)}(f)-
\mathcal{B}^2_{\lambda,I(x)}(W))^2\dx.
\]
\end{enumerate}
This condition holds when $\lambda=0$.\\

Let
\[
F^k_+=\dot{H}^k(\mathbb{T}^d,\Rr)\times H^k_+(\mathbb{T}^d,\Rr)\times L^2(\mathbb{T}^d,\Rr^d)\times\Rr,
\]
by a classical solution to \eqref{mainl} we mean a tuple
$(u_{\lambda},m_{\lambda},V_{\lambda},\overline{H}_{\lambda})\in\bigcap\limits_k F^k_+$.
\begin{teo}
\label{maint}
Assume the Assumptions  (B\ref{Hcont})-(B\ref{monotcond}) hold. Furthermore, suppose that
 either
\begin{itemize}
\item[A.]
 (A\ref{qzv})-(A\ref{bdv2}) and  one of the following assumptions is satisfied:
\begin{itemize}
\item[(i)] (A\ref{g}\ref{lnm}
\item[(ii)]  (A\ref{g}\ref{mgama}, with any $\gamma>0$ if $d\leq 4,$ and $\gamma<\frac{1}{d-4}$ if $d\geq 5.$
\end{itemize}
\end{itemize}
or
\begin{itemize}
\item[B.]
 (H1),(H2) holds  with $\epsilon\in[0,\epsilon_0]$,  (A\ref{g}\ref{lnm} and $d\leq 3$.
\end{itemize}
Then there exists a classical solution $(u,m,V,\overline{H})$ to \eqref{main}.
\end{teo}
\begin{proof}
For big enough $k$ we can define
$
E\colon\Rr\times F^k_+\to F^{k-2}
$
by
\[
E(\lambda,u,m,V,\overline{H})=\left(
\begin{array}{c}
\Delta m-\div(D_pH_{\lambda}(x,Du,m,V)m)\\
-\Delta u+H_{\lambda}(x,Du,m,V)+\overline{H}\\
V-D_pH_{\lambda}(x,Du,m,V)\\
-\int_{\mathbb{T}^d}m\dx+1
\end{array}
\right).
\]
Then \eqref{mainl} can be written as $E(\lambda,u_{\lambda},m_{\lambda},V_{\lambda},\overline{H}_{\lambda})=0.$
The partial derivative of $E$ at a point $v_{\lambda}=(u_{\lambda},m_{\lambda},V_{\lambda},\overline{H}_{\lambda})$
\[
\mathcal{L}_{\lambda}:=D_2E(\lambda,v_{\lambda})\colon F^k\to F^{k-2},
\]
is given by
\[
\mathcal{L}_{\lambda}(w)(x)=\left(
\begin{array}{c}
\Delta f(x)-\div(V_{\lambda}(x)f(x))-\div([\mathcal{B}^0_{\lambda,I_{\lambda}(x)}(D\psi(x))+
\mathcal{B}^1_{\lambda,I_{\lambda}(x)}(f)+
\mathcal{B}^2_{\lambda,I_{\lambda}(x)}(W)]m_{\lambda}(x))\\
-\Delta \psi(x)-\mathcal{A}^0_{\lambda,I_{\lambda}(x)}(D\psi(x))-\mathcal{A}^1_{\lambda_0,I_{\lambda}(x)}(f)+g'(m_{\lambda_0}(x))f(x)-
\mathcal{A}^2_{\lambda,I_{\lambda}(x)}(W)+h\\
W(x)-\mathcal{B}^0_{\lambda,I_{\lambda}(x)}(D\psi(x))-
\mathcal{B}^1_{\lambda,I_{\lambda}(x)}(f)-
\mathcal{B}^2_{\lambda,I_{\lambda}(x)}(W)\\
-\int_{\mathbb{T}^d}f\dx
\end{array}
\right),
\]
where $I(x)=(Du_{\lambda}(x),x,m_{\lambda},V_{\lambda})$ and $w=(\psi,f,W,h)\in F^k$. Note that $\mathcal{L}_{\lambda}$ is well defined for any $k>1.$

Note also that for a classical solution $(u_{\lambda},m_{\lambda},V_{\lambda},\overline{H}_{\lambda})$ to \eqref{mainl}, we get from the third equation that $V_{\lambda}\in\bigcap\limits_k F_k$. Define the set
\[
\Lambda=\{\,\lambda|\quad 0\leq\lambda\leq 1,\,\eqref{mainl} \text{ has a classical solution } (u_{\lambda},m_{\lambda}, V_{\lambda}, \Hh_{\lambda})\,\}.
\]

Note that $0\in\Lambda$, with $(u_0,m_0,V_0, \Hh_0)\equiv(0,1,0,-g(1)).$ Our purpose is to prove $\Lambda=[0,1].$
Let $\lambda_k\in [0,1]$, $\lambda_k\to\lambda_0.$ It is easy to see that the Assumptions $(A1)-(A11)$ for $H$ imply the
corresponding properties for $H_{\lambda}$ with uniform constants for any $\lambda\in[0,1].$ Thus the results of the previous sections (Theorem \ref{usmooth}) and Sobolev's embedding theorems imply that we can bound uniformly derivatives of any order of the solutions $u_{\lambda_k},m_{\lambda_k}$, and also the $C^1$ norm of $V_{\lambda_k}$. Thus we can assume that there exist functions $u,m,V$ and a number $\overline{H}$, such that $u_{\lambda_k}\to u$,$m_{\lambda_k}\to m$ in $H^l(\mathbb{T}^d)$ for every integer $l$, and hence in $C^l(\mathbb{T}^d)$ for every $l$, and also $V_{\lambda_k}\to V$ in $L^2(\mathbb{T}^d)$ and $\overline{H}_{\lambda_k}\to \overline{H}$. Passing to the limit in \eqref{mainl} for $\lambda=\lambda_k$ and using Assumption (B\ref{Hcont}) we get that $(u,m,V,H)$ is a  classical solution to \eqref{mainl} for $\lambda=\lambda_0.$ From $m_{\lambda_k}\geq \bar{m}$ we have $m>0.$ This proves $\lambda_0\in\Lambda$, thus $\Lambda$ is closed.
To prove that $\Lambda$ is open we need to prove that $\mathcal{L}_{\lambda}$ is invertible in order to use an implicit function theorem. For this let $F=F^1$. For $w_1,w_2 \in F$ with smooth components we can define

\[
B_{\lambda}[w_1,w_2]=\int_{\mathbb{T}^d} w_2\cdot \mathcal{L}_{\lambda}(w_1).
\]
Using integration by parts we have for $w_1,w_2$ smooth,
\begin{equation}
\begin{split}
B_{\lambda}[w_1,w_2]=\int_{\mathbb{T}^d}
[&m_{\lambda}\mathcal{B}^0_{\lambda,I_{\lambda}(x)}(D\psi_1)D\psi_2+
m_{\lambda}\mathcal{B}^1_{\lambda,I_{\lambda}(x)}(f_1)D\psi_2+
m_{\lambda}\mathcal{B}^2_{\lambda,I_{\lambda}(x)}(W_1)D\psi_2\\&+\mathcal{A}^0_{\lambda,I_{\lambda}(x)}(D\psi_2)f_1
-\mathcal{A}^0_{\lambda,I_{\lambda}(x)}(D\psi_1)f_2+g'(m_{\lambda})f_1f_2-\mathcal{A}^1_{\lambda,I_{\lambda}(x)}(f_1)f_2\\&-
\mathcal{A}^2_{\lambda,I_{\lambda}(x)}(W_1)f_2+D\psi_1Df_2-Df_1D\psi_2
+h_1f_2 -h_2f_1+W_1W_2\\&-\mathcal{B}^0_{\lambda,I_{\lambda}(x)}(D\psi_1)W_2-
\mathcal{B}^1_{\lambda,I_{\lambda}(x)}(f_1)W_2-
\mathcal{B}^2_{\lambda,I_{\lambda}(x)}(W_1)W_2]\dx.
\end{split}
\end{equation}
This last expression is well defined on $F\times F.$ Thus it defines a bilinear form $B_{\lambda}\colon F\times F\to \Rr$.
\begin{clm}
$B$ is bounded $|B_{\lambda}[w_1,w_2]|\leq C\|w_1\|_F\|w_2\|_F$.
\end{clm}
We use the Assumption (B\ref{HvHmdb}) and Holder's inequality on each summand.
\begin{clm}
There exists a linear bounded mapping $A\colon F\to F$ such that $B_{\lambda}[w_1,w_2]=(Aw_1,w_2)_F$.
\end{clm}
For each fixed element $w\in F$, the operator $w_1\mapsto B_{\lambda}[w_1,w]$ is a bounded linear functional on $F$; whence the Riesz Representation Theorem ensures the existence of a unique element $\nu_1\in F$ such that
\[
B[w_1,w]=(\nu_1,w)_F, \text{ for all } w\in H.
\]
Let us define the operator $A\colon F \to F$ by $Aw_1=\nu_1$, so
\[
B_{\lambda}[w_1,w_2]=(Aw_1,w_2)\quad (w_1,w_2\in F).
\]
It is easy to see that $A$ is linear. Furthermore
\[
\|Aw_1\|^2_F=(Aw_1,Aw_1)=B_{\lambda}[w_1,Aw_1]\leq C\|w_1\|_F\|Aw_1\|_F.
\]
Thus $\|Aw_1\|_F\leq C\|w_1\|_F$, and so $A$ is bounded.
\begin{clm}
There exists a positive constant $c$ such that $\|Aw\|_F\geq c\|w\|_F$ for all $w\in F.$
\end{clm}
If the previous claim were false there would exist a sequence $w_n\in F$ with $\|w_n\|_F=1$ such that $Aw_n\to 0.$ Let $w_n=(\psi_n,f_n,W_n,h_n)$ and $\tilde{w}_n=(0,0,\tilde{W}_n,0)$ where
\[
\tilde{W}_n(x)=W_n(x)-\mathcal{B}^0_{\lambda,I_{\lambda}(x)}(D\psi_n(x))-
\mathcal{B}^1_{\lambda,I_{\lambda}(x)}(f_n)-
\mathcal{B}^2_{\lambda,I_{\lambda}(x)}(W_n).
\]
Assumption (B\ref{HvHmdb}) gives $\|\tilde{w}_n\|_F\leq C\|w_n\|_F=C$, thus we have
\[
\|\tilde{W}_n\|^2_{L^2}=B_{\lambda}[w_n,\tilde{w}_n]=(Aw_n,\tilde{w}_n)\to 0.
\]
Hence $\tilde{W}_n\to 0$ in $L^2.$
Let now $\bar{w}_n=(\psi_n,f_n,0,h_n)$ then, using Assumption (B\ref{monotcond}),
\[
-C\|\tilde{W}_n\|^2_{L^2}+ \theta\int_{\mathbb{T}^d}\bar{m}|D\psi_n|^2+|f_n|^2\dx\leq \mathcal{H}_{\lambda, I_{\lambda}}(D\psi_n,f_n,W_n)=B_{\lambda}[w_n,\bar{w}_n]\to 0.
\]
Thus $\psi_n\to 0$ in $\dot{H}^1_0$ and $f_n\to 0$ in $L^2$. This, combined with $\tilde{W}_n\to 0$, implies $W_n\to 0$. Taking $\check{w}_n=(f_n-\int f_n,0,0,0)\in F$ we get
\begin{equation*}
\begin{split}
&\int_{\mathbb{T}^d}[-|Df_n|^2+m_{\lambda}\mathcal{B}^0_{\lambda,I_{\lambda}(x)}(D\psi_n)Df_n+
m_{\lambda}\mathcal{B}^1_{\lambda,I_{\lambda}(x)}(f_n)Df_n+\\
&m_{\lambda}\mathcal{B}^2_{\lambda,I_{\lambda}(x)}(W_n)Df_n+\mathcal{A}^0_{\lambda,I_{\lambda}(x)}(Df_n)f_n]\dx=B[w_n,\check{w}_n]
=(Aw_n,\check{w}_n),
\end{split}
\end{equation*}
using the expressions for $\mathcal{A}^0_{\lambda,I_{\lambda}(x)}, \mathcal{B}^0_{\lambda,I_{\lambda}(x)}$, Assumption (B\ref{HvHmdb}) and Cauchy's inequality we get
\[
\frac{1}{2}\|Df_n\|^2_{L^2(\mathbb{T}^d)}
-C\left(\|D\psi_n\|^2_{L^2(\mathbb{T}^d)}+\|f_n\|^2_{L^2(\mathbb{T}^d)}+\|W_n\|^2_{L^2(\mathbb{T}^d,\Rr^d)}\right)\leq-(Aw_n,\check{w}_n)\to 0,
\]
were $C$ depends only on $u_{\lambda},m_{\lambda},V_{\lambda}$ and $H_{\lambda}$, thus since $D\psi_n,f_n,W_n\to 0$ in $L^2$ we get that $f_n \to 0$ in $H^1(\mathbb{T}^d).$
Now taking $\breve{w}=(0,1,0,0)$ we get
\[
\int_{\mathbb{T}^d}[-\mathcal{A}^0_{\lambda,I_{\lambda}(x)}(D\psi_n)+g'(m_{\lambda})f_n-\mathcal{A}^1_{\lambda,I_{\lambda}(x)}(f_n)-
\mathcal{A}^2_{\lambda,I_{\lambda}(x)}(W_n)]\dx+h_n=B[w_n,\breve{w}]
=(Aw_n,\breve{w})\to 0,
\]
using the expressions for $\mathcal{A}^0_{\lambda,I_{\lambda}(x)}, \mathcal{B}^0_{\lambda,I_{\lambda}(x)}$, the Assumption (B\ref{HvHmdb}) and the fact that $D\psi_n,f_n,W_n\to 0$ in $L^2$ we get $h_n\to 0$. We conclude that $w_n\to 0$, which contradicts with $\|w_n\|_F=1$.
\begin{clm}
$R(A)$ is closed in $F$.
\end{clm}
If $Au_n\to w$ in $F$ then $c\|u_n-u_m\|_F\leq \|Au_n-Au_m\|_F\to 0$ as $n,m\to \infty$. Therefore
 $u_n$ converges to some $u\in F$, then $Au=w$ this proves that $R(A)$ is closed.
\begin{clm}
$R(A)=F$.\end{clm}
Suppose $R(A)\neq F$, then since $R(A)$ is closed in $F$ there exists $w\neq 0$ such that $w\bot R(A)$ in $F$. Let $w=(\psi,f,W,h)$, take $\tilde{w}=(\psi,f,\tilde{W},h)$ where $\tilde{W}$ is given by
\[
\tilde{W}(x)=\mathcal{B}_{\lambda,I_{\lambda}(x)}(D\psi(x))+
\mathcal{B}^1_{\lambda,I_{\lambda}(x)}(f)+
\mathcal{B}^2_{\lambda,I_{\lambda}(x)}(\tilde{W}),
\]
such $\tilde{W}$ exists since the operator $Id-\mathcal{B}^2_{\lambda}$ is invertible.
Then
\[
0=(A\tilde{w},w)=B_{\lambda}[\tilde{w},w]=\mathcal{H}_{\lambda,I_{\lambda}}(f,D\psi,\tilde{W})\geq
\theta\int_{\mathbb{T}^d}\bar{m}|D\psi|^2+|f|^2\dx
\]
thus $\psi=0,f=0$.
Then, let $\hat{w}=(0,0,\hat{W},0)$ where we take
\[
\hat{W}-\mathcal{B}^2_{\lambda,I_{\lambda}(x)}(\hat{W})=W
\]
using the invertibility of operator $Id-\mathcal{B}^2_{\lambda}$.
This gives $\|W\|_{L^2}^2=B_{\lambda}[\hat{w}, w]=(A\hat{w}, w)=0$.
Choosing now $\bar{w}=(0,1,0,0)$  gives
$h=B_{\lambda}[\bar{w},w]=(A\bar w, w)=0$.
Thus $w=0$ and this implies $R(A)=F$.

\begin{clm}
For any $w_0\in F^0$ there exists a unique $w\in F$ such that $B_{\lambda}[w,\tilde{w}]=
(w_0,\tilde{w})_{F^0}$ for all $\tilde{w}\in F.$ This implies that $w$ is a unique weak solution to the equation $\mathcal{L}_{\lambda}(w)=w_0$.
Then regularity theory implies that $w\in F^2$ and
$\mathcal{L}_{\lambda}(w)=w_0$ in the sense of $F^2.$
\end{clm}

%\begin{proof}
Consider the functional $\tilde{w}\mapsto(w_0,\tilde{w})_{F^0}$ on $F$. By Riesz representation theorem, there exists $\omega\in F$ such that $(w_0,\tilde{w})_{F^0}=(\omega,\tilde{w})_F$ now taking $w=A^{-1}\omega$ we get $$B[w,\tilde{w}]=(Aw,\tilde{w})_F=(\omega,\tilde{w})_F=(w_0,\tilde{w})_{F^0}.$$
Let $w=(\psi,f,W,h)$ and $w_0=(\psi_0,f_0,W_0,h_0)$, taking $\tilde{w}=(\tilde{\psi},0,0,0),(0,\tilde{f},0,0),(0,0,\tilde{W},0)$, and $(0,0,0,1)$, we get, respectively,
\begin{equation}
\label{1wk}
\int_{\mathbb{T}^d}
m_{\lambda}\mathcal{B}^0_{\lambda,I_{\lambda}(x)}(D\psi)D\tilde{\psi}+
m_{\lambda}\mathcal{B}^1_{\lambda,I_{\lambda}(x)}(f)D\tilde{\psi}+
m_{\lambda}\mathcal{B}^2_{\lambda,I_{\lambda}(x)}(W)D\tilde{\psi}
+f V_\lambda  D\tilde \psi
-DfD\tilde{\psi}=\int_{\mathbb{T}^d}\psi_0\tilde{\psi}
\end{equation}

\begin{equation}
\label{2wk}
\begin{split}
&\int_{\mathbb{T}^d}[
%\mathcal{A}_{\lambda,I_{\lambda}(x)}(D\psi)\tilde{f}
-\mathcal{A}^0_{\lambda,I_{\lambda}(x)}(D\psi)\tilde f+g'(m)f\tilde{f}-\mathcal{A}^1_{\lambda,I_{\lambda}(x)}(f)\tilde{f}\\&-
\mathcal{A}^2_{\lambda,I_{\lambda}(x)}(W)\tilde{f}+D\psi D\tilde{f}+h\tilde{f}]\dx
=\int_{\mathbb{T}^d}f_0\tilde{f},
\end{split}
\end{equation}
\[
\int_{\mathbb{T}^d} W\tilde{W}-\mathcal{B}^0_{\lambda,I_{\lambda}(x)}(D\psi)\tilde{W}-
\mathcal{B}^1_{\lambda,I_{\lambda}(x)}(f)\tilde{W}-
\mathcal{B}^2_{\lambda,I}(W)\tilde{W}=\int_{\mathbb{T}^d}W_0\tilde{W},
\]
and
\[
-\int_{\mathbb{T}^d}f=h_0.
\]
Since we can take $\tilde{\psi},\tilde{f}\in H^1(\mathbb{T}^d,\Rr)$ and $\tilde{W}\in L^2(\mathbb{T}^d,\Rr^d)$  arbitrarily, we get
\[W(x)=\mathcal{B}_{\lambda,I_{\lambda}(x)}(D\psi)+
\mathcal{B}^1_{\lambda,I_{\lambda}(x)}(f)+
\mathcal{B}^2_{\lambda,I_{\lambda}(x)}(W)\]
then the equation \eqref{1wk} gives that $f$ is a weak solution to

\[
\Delta f-\div(V_{\lambda}f)-\div(Wm_{\lambda})=\psi_0
\]
and\eqref{2wk} means that $\psi$ is a weak solution to

\[
\Delta \psi+\mathcal{A}_{\lambda,I_{\lambda}(x)}(D\psi)+\mathcal{A}^1_{\lambda,I_{\lambda}(x)}(f)-g'(m_{\lambda})f+
\mathcal{A}^2_{\lambda,I_{\lambda}(x)}(W)-h=f_0.
\]
The last equation gives $\Delta \psi\in L^2$ thus $\psi\in H^2$, then the equation for $W$ yields that $W\in H^1$ and the equation for $f$ gives $\Delta f\in L^2$ hence $f\in H^2$. We conclude that $w=(\psi,f,W,h)\in F^2$ and $\mathcal{L}_{\lambda}(w)=w_0.$
%\end{proof}

This implies that $\mathcal{L}_{\lambda}$ is bijective operator from $F^2$ to $F^0$. Then $\mathcal{L}_{\lambda}$ it is  injective as an operator from $F^k$ to $F^{k-2}$ for any $k\geq 2.$ To prove that it is also surjective take any $w_0\in F^{k-2}$, then there exists $w\in F^2$ such that $\mathcal{L}_{\lambda}(w)=w_0$. Using a bootstrap argument like the one in the proof of the previous lemma we conclude that in fact $w\in F^k$. This proves that $\mathcal{L}_{\lambda}\colon F^k\to F^{k-2}$ is surjective and therefore also bijective.

\begin{clm}
$\mathcal{L}_{\lambda}$ is an isomorphism from $F^k$ to $F^{k-2}$ for any $k\geq 2.$
\end{clm}
%\begin{proof}
Since we have $\mathcal{L}\colon F^k \to F^{k-2}$ is bijective we just need to prove that it is also bounded. But that follows directly from the Assumptions (B\ref{HvHmreg}), and (B\ref{HvHmdb}).
%\end{proof}

\begin{clm}
We now prove that the set $\Lambda$ is open.
\end{clm}

Indeed for a point $\lambda_0\in \Lambda$ we have proven that the partial derivative $\mathcal{L}=D_2E(\lambda_0,v_{\lambda_0})\colon F^k \to F^{k-2}$ is an isometry for every $k$. Hence by the implicit function theorem
(see \cite{D1}) there exists a unique solution $v_{\lambda}\in F^k_+$ to $E(\lambda,v_{\lambda})=0$ for some neighborhood $U$ of $\lambda_0$. By the uniqueness these solutions coincide with each other for all $k$, thus there exists a solution $v_{\lambda}$ which belongs to all $F^k_+$, hence is a classical solution. Thus we conclude that $U\subset\Lambda$, which proves that $\Lambda$ is open. We have proven that $\Lambda$ is both open and closed, hence $\Lambda=[0,1]$.
\end{proof}

\section{Examples}
\label{exmpl}
\subsection{Velocity independent Hamiltonians}

In this section we consider an Hamiltonian that does not depend on the velocity field:
\[
H\colon\mathbb{T}^d\times\mathbb{R}^d\times \mathcal{P}^{ac}(\mathbb{T}^d)\to\Rr.
\]
And we assume that it can be extended to a function
\[
H\colon\mathbb{T}^d\times\mathbb{R}^d\times H^1(\mathbb{T}^d,\Rr)\to\Rr.
\]
The system \eqref{main} in this case is
\begin{equation}
\label{main*}
\begin{cases}
\Delta u(x) +H(x,Du(x),m(x))=\overline{H}\\
\Delta m(x)- \div(D_pH(x,Du(x),m(x))m(x))=0.
\end{cases}
\end{equation}
Let
\begin{equation*}
H_{\lambda}(x,p,m)=\lambda H(x,p,m) +(1-\lambda)\left[\frac{|p|^2}{2}-g(m)\right],\quad \lambda\in[0,1],
\end{equation*}
and consider the corresponding equations
\begin{equation}
\label{main*l}
\begin{cases}
\Delta u_{\lambda}(x) +H_{\lambda}(x,Du_{\lambda}(x),m_{\lambda}(x))=\overline{H}_{\lambda}\\
\Delta m_{\lambda}(x)- \div(D_pH_{\lambda}(x,Du_{\lambda}(x),m_{\lambda}(x))m_{\lambda}(x))=0.
\end{cases}
\end{equation}

\renewcommand{\labelenumi}{({\bf C\arabic{enumi}})}
\begin{enumerate}
\item
\label{c0}
We assume that functions $H(x,p,m), D^2_{pp}H(x,p,m), D^2_{px}H(x,p,m)$ are continuous in $m$ with respect to the uniform convergence.
\item
\label{c1}
\label{HvHmreg*}
We assume that for any $f\in H^k(\mathbb{T}^d,\Rr)$ the functions

$D_mH(x,p,m)(f)+g'(m(x))f(x)$ and
$D^2_{pm}H(x,p,m)(f),$ are smooth in $x$ and $p$.
\item
\label{c2}
For any positive integer $l$, $m\in H^k_+(\mathbb{T}^d,\Rr)$ and any number $R>0$, there exists a constant $C(l,m,R)$ such that
\[
\left|D^l_{x,p}[D_mH(x,p,m)(f)+g'(m(x))f(x)]\right|\leq C(l,m,R)\|f\|_{L^2},\]
and
\[
\left|D^l_{x,p}[D^2_{pm}H(x,p,m)(f)]\right|\leq C(l,m,R)\|f\|_{L^2},
\]
for all $ x\in\mathbb{T}^d, |p|\leq R$ and $f\in L^2(\mathbb{T}^d,\Rr)$.
\item
\label{c3}
There exists $\theta>0$ such that for any $(u_{\lambda},m_{\lambda})$ solution to \eqref{main*l}, and any $(Q,f)\in H^k(\mathbb{T}^d,\Rr^d)\times H^k(\mathbb{T}^d,\Rr)$ we have
\begin{equation*}
\begin{split}
&\int_{\mathbb{T}^d}[m_{\lambda}(x)Q(x)D^2_{pp}H(x,Du_{\lambda}(x),m_{\lambda})Q(x)+
m_{\lambda}(x)Q(x)D^2_{pm}H(x,Du_{\lambda}(x),m_{\lambda})(f)
\\&-f(x)D_{m}H(x,Du_{\lambda}(x),m_{\lambda})(f)]\dx\geq \theta\int_{\mathbb{T}^d}m_{\lambda}|Q|^2+|f(x)|^2\dx.
\end{split}
\end{equation*}
\end{enumerate}

\begin{teo}
Assume the Assumptions (A\ref{qzv})-(A\ref{bdv2}), (C\ref{c0})-(C\ref{c3}) hold. Furthermore, suppose that  one of the following assumptions is satisfied:
\begin{itemize}
\item[(i)] (A\ref{g}\ref{lnm}
\item[(ii)]  (A\ref{g}\ref{mgama}, with any $\gamma>0$ if $d\leq 4,$ and $\gamma<\frac{1}{d-4}$ if $d\geq 5.$
\end{itemize}
Then there exists a classical solution $(u,m,\overline{H})$ to \eqref{main*}.
\end{teo}

\begin{proof}
We just need to check Assumptions (B\ref{Hcont})-(B\ref{monotcond}) so that we can apply Theorem \ref{maint}.
For a point $I=(x,p,m)\in\Rr^d\times\mathbb{T}^d \times H^k_+(\mathbb{T}^d,\Rr)$ we have $\mathcal{A}^0_{\lambda,I}\colon\Rr^d\to \Rr,\mathcal{B}^0_{\lambda,I}\colon\Rr^d\to \Rr^d$,  given by

\[
\mathcal{A}^0_{\lambda,I}(w)=\lambda D_pH(x,p,m)\cdot w+(1-\lambda)p\cdot w,\quad
\mathcal{B}^0_{\lambda,I}(w)=\lambda D^2_{pp}H(x,p,m)w +(1-\lambda)w,
\]
\[
\mathcal{A}^1_{\lambda,I}(f)=\lambda[D_mH(x,p,m)(f)+g'(m(x))f(x)],\quad
\mathcal{B}^1_{\lambda,I}(f)=\lambda D^2_{pm}H(x,p,m)(f),
\]
and since there is not velocity field:
$
\mathcal{A}^2_{\lambda,I}(W)=
\mathcal{B}^2_{\lambda,I}(W)=0.
$
From this  Assumption ($B$\ref{invHpv}) holds automatically. Assumptions (B\ref{Hcont})-(B\ref{HvHmdb}) then follow easily form Assumptions (C\ref{c0})-(C\ref{c2}).
For a point $I(x)=(x,Du_{\lambda}(x),m_{\lambda},V)$ where $u_{\lambda},m_{\lambda}$ is a solution to \eqref{main*l} we have
\begin{equation*}
\begin{split}
\mathcal{H}_{\lambda,I}(Q,f,W)=\int_{\mathbb{T}^d}[&m_{\lambda}Q\mathcal{B}^0_{\lambda,I(x)}(Q)+
m_{\lambda}Q\mathcal{B}^1_{\lambda,I(x)}(f)+g'(m_{\lambda})f^2-f\mathcal{A}^1_{\lambda,I(x)}(f)]\dx
\end{split}
\end{equation*}
Note that because $\mathcal{H}_{\lambda, I}$ does not depend on $W$, it is enough for the Assumption (B\ref{monotcond}) to hold to check that
\begin{equation}
\begin{split}
\label{cccd}
\mathcal{H}_{\lambda, I}(Q,f,W)\geq \theta\int_{\mathbb{T}^d}m_{\lambda}|Q|^2+|f(x)|^2\dx.
\end{split}
\end{equation}
We have $\mathcal{H}_{\lambda, I}(Q,f,W)=\lambda\mathcal{H}_{1, I}(Q,f,W)+(1-\lambda)\mathcal{H}_{0, I}(Q,f,W),$ and
\[
\mathcal{H}_{0, I}(Q,f,W)=\int_{\mathbb{T}^d}m_{\lambda}|Q|^2+g'(m_{\lambda})|f(x)|^2\dx\geq\theta_0\int_{\mathbb{T}^d}m_{\lambda}|Q|^2+|f(x)|^2\dx,
\]
since for any solution $(u_{\lambda},m_{\lambda})$ to \eqref{main*l} $m_{\lambda}\in[\bar{m},\overline{C}]$ and $g$ is strictly increasing so $g'(m_{\lambda})\geq\eta_0$ for some constant $\eta_0>0$, and we take $\theta_0=\min\{1,\eta_0\}$. Then it is enough to check the condition \ref{cccd} just for $\lambda=1$.
Hence the condition \eqref{cccd} is equivalent to
\begin{equation*}
\begin{split}
&\mathcal{H}_{1,I}(Q,f,W)=\int_{\mathbb{T}^d}[m_{\lambda}(x)Q(x)D^2_{pp}H(x,Du_{\lambda}(x),m_{\lambda})Q(x)+
m_{\lambda}(x)Q(x)D^2_{pm}H(x,Du_{\lambda}(x),m_{\lambda})(f)
\\&-f(x)D_{m}H(x,Du_{\lambda}(x),m_{\lambda})(f)]\dx\geq \theta\int_{\mathbb{T}^d}m_{\lambda}|Q|^2+|f(x)|^2\dx,
\end{split}
\end{equation*}
that is (C\ref{c3}).
\end{proof}
Note that the Assumption (C\ref{c3}) can be interpreted in some sense as a operator inequality
analog to the condition
\[\left[
\begin{array}{cc}
	mD^2_{pp}H& \frac{1}{2}D^2_{pm}H\\
	\frac{1}{2}D^2_{pm}H&-D_{m}H
\end{array}
\right]\geq \theta I,
\]
obtained by Lions for the uniqueness of mean-field games with local dependence
(see \cite{Gueant2}).

\subsection{Velocity dependent example}
In this section we consider the following type of Hamiltonians:
\begin{equation}
\label{simplehamiltonian}
H(x,p,m,V)=h(x,p)+\alpha p\int_{\mathbb{T}^d} V m\dy-g(m).
\end{equation}
Where $h:\mathbb{T}^d\times\colon\Rr^d \to\Rr$ satisfies the following assumptions
\renewcommand{\labelenumi}{({\bf D\arabic{enumi}})}
\begin{enumerate}
\item
\label{d1}
$h$ is smooth in $x,p.$
\item
\label{d2}
$|p|^2\leq C+ C|h|$.
\item
\label{d3}
$p\cdot D_ph-h\geq ch-C$.
\item
\label{d4}
$|D_ph|^2\leq C+Ch$.
\item
\label{d5}
$D^2_{pp}h\geq\sigma I$, for some $\sigma>0$, where $I$ is the $d$-dimensional identity matrix.
\item
\label{d6}
$|D^2_{xx}h|, |D^2_{xp}h|^2\leq C+Ch$.
\item
\label{d7}
$h\leq C+C|p|^2$.
\item
\label{d8}
$|D_xh|\leq C+C|p|^\beta$, with $ 0\leq\beta<2$.

\end{enumerate}
For some constants $c,C>0.$
We have
\begin{equation}
H_{\lambda}(x,p,m,V)=h_{\lambda}(x,p)+\lambda\alpha p\int_{\mathbb{T}^d} V m\dy-g(m),\quad \lambda\in[0,1].
\end{equation}
where $h_{\lambda}(x,p)=\lambda h(x,p)+(1-\lambda)\frac{|p|^2}{2}$.

\begin{teo}
Assume the Assumptions (D\ref{d1})-(D\ref{d8}) hold. Furthermore, suppose that  one of the following assumptions is satisfied:
\begin{itemize}
\item[(i)] (A\ref{g}\ref{lnm}
\item[(ii)]  (A\ref{g}\ref{mgama}, with any $\gamma>0$ if $d\leq 4,$ and $\gamma<\frac{1}{d-4}$ if $d\geq 5.$
\end{itemize}
Then there there exists $\alpha_0>0$ such that for any $\alpha,|\alpha|\leq \alpha_0$ there exists a classical solution $(u,m,V,\overline{H})$ to \eqref{main}.
\end{teo}

\begin{proof}
Assumptions (D\ref{d1})-(D\ref{d8}) imply easily (A\ref{qzv})-(A\ref{bdv2}) with $\delta=\frac{\alpha^2}{2}$, so it is enough to check the Assumptions (B\ref{Hcont})-(B\ref{monotcond}). Assumption (B\ref{Hcont}) easily follows from \eqref{simplehamiltonian}.
Now we proceed to checking Assumptions (B\ref{HvHmreg})-(B\ref{monotcond}), let $I=(x,p,m,V)\in\mathbb{T}^d \times\Rr^d\times H^k_+(\mathbb{T}^d,\Rr)\times L^2(\mathbb{T}^d,\Rr^d)$, we have:
%that $\mathcal{A}^0_{\lambda,I}\colon\Rr^d\to \Rr,\mathcal{B}^0_{\lambda,I}\colon\Rr^d\to \Rr^d$ are
%\vskip0.3cm
\[
\mathcal{A}^0_{\lambda,I}(w)=D_ph_{\lambda}(x,p)\cdot w+\lambda\alpha\int_{\mathbb{T}^d} V m\dy\cdot w ,\quad
\mathcal{B}^0_{\lambda,I}(w)=D_{pp}h_{\lambda}w,
\]
\[
\mathcal{A}^1_{\lambda,I}(f)=\lambda\alpha p\cdot\int_{\mathbb{T}^d} V f\dy,\quad
\mathcal{B}^1_{\lambda,I}(f)=\lambda\alpha \int_{\mathbb{T}^d} V f\dy,
\]

\[
\mathcal{A}^2_{\lambda,I}(W)=\lambda\alpha p\cdot\int_{\mathbb{T}^d} Wm\dy,\quad
\mathcal{B}^2_{\lambda,I}(W)=\lambda\alpha \int_{\mathbb{T}^d} Wm\dy.
\]
For a fixed $f\in H^k(\mathbb{T}^d,\Rr)$ and $W\in L^2(\mathbb{T}^d,\Rr^d)$ the functions $\mathcal{A}^1_{\lambda,I}(f),$ $\mathcal{B}^1_{\lambda,I}(f),$ $\mathcal{A}^2_{\lambda,I}(W),$ $\mathcal{B}^2_{\lambda,I}(W)$ do not depend on $x$ and are either linear in $p$ or a constant.  For all $ x\in\mathbb{T}^d, |p|\leq R,$ we have
\[
|\mathcal{A}^1_{\lambda}(f)|\leq R|\alpha|\left|\int_{\mathbb{T}^d} Vf\dy\right|\leq R|\alpha|\|V\|_{L^2(\mathbb{T}^d)}\|f\|_{L^2(\mathbb{T}^d)}
\]

\[
|D_p[\mathcal{A}^1_{\lambda}(f)]| \leq|\alpha|\left|\int_{\mathbb{T}^d} Vf\dy\right|\leq |\alpha|\|V\|_{L^2(\mathbb{T}^d)}\|f\|_{L^2(\mathbb{T}^d)}
\]

\[
|\mathcal{B}^1_{\lambda}(f)| \leq|\alpha|\left|\int_{\mathbb{T}^d} Vf\dy\right|\leq |\alpha|\|V\|_{L^2(\mathbb{T}^d)}\|f\|_{L^2(\mathbb{T}^d)}
\]

\[
|\mathcal{A}^2_{\lambda}(W)|\leq R|\alpha|\left|\int_{\mathbb{T}^d} Wm\dy\right|\leq R|\alpha|\|m\|_{L^2(\mathbb{T}^d)}\|W\|_{L^2(\mathbb{T}^d)}
 \]

\[
|D_p\mathcal{A}^2_{\lambda}(W)|\leq|\alpha|\left|\int_{\mathbb{T}^d} Wm\dy\right|\leq |\alpha|\|m\|_{L^2(\mathbb{T}^d)}\|W\|_{L^2(\mathbb{T}^d)}
\]

\[
|\mathcal{B}^2_{\lambda}(W)|\leq|\alpha|\left|\int_{\mathbb{T}^d} Wm\dy\right|\leq R|\alpha|\|m\|_{L^2(\mathbb{T}^d)}\|W\|_{L^2(\mathbb{T}^d)}
\]
So the Assumptions (B\ref{HvHmdb}), (B\ref{HvHmreg}) hold.

To check Assumption (B\ref{invHpv}), take a point $(P,m,V)\in C^{\infty}(\mathbb{T}^d,\Rr^d)\times C^{\infty}(\mathbb{T}^d,\Rr)\times L^2(\mathbb{T}^d,\Rr^d),$ the operators
$\mathcal{A}^1_{\lambda},\mathcal{B}^1_{\lambda}\colon H^k(\mathbb{T}^d,\Rr)\to C^{\infty}(\mathbb{T}^d,\Rr)$
are given by
\[
\mathcal{A}^1_{\lambda}(f)(x)=\lambda\alpha P(x)\cdot\int_{\mathbb{T}^d} V f\dy,\quad
\mathcal{B}^1_{\lambda}(f)(x)=\lambda\alpha \int_{\mathbb{T}^d} V f\dy,
\]
where $I(x)=(x,P(x),m,V)$.
Similarly, the operators $\mathcal{A}^2_{\lambda},\mathcal{B}^2_{\lambda}\colon L^2(\mathbb{T}^d,\Rr^d)\to C^{\infty}(\mathbb{T}^d,\Rr^d)$
are given by
\[
\mathcal{A}^2_{\lambda}(W)(x)=\mathcal{A}^2_{\lambda,I(x)}(W)=\lambda\alpha P(x)\cdot\int_{\mathbb{T}^d} W m\dy,\quad
\mathcal{B}^2_{\lambda}(W)(x)=\mathcal{B}^2_{\lambda,I(x)}(W)=\lambda\alpha \int_{\mathbb{T}^d} W m\dy.
\]
Let $\beta=\lambda\alpha$, and assume $|\alpha|<1$, then also $|\beta|<1$ and we have
\[
(Id-\mathcal{B}^2_{\lambda})(W)=W-\beta \int_{\mathbb{T}^d} W m\dy:=\widetilde{W}
\]
Integrating this with respect to $m$ we get $(1-\beta)\int_{\mathbb{T}^d} W m\dy=\int_{\mathbb{T}^d} \widetilde{W} m\dy$, hence
\vskip0.5cm
$W=\widetilde{W}+\frac{\beta}{1-\beta}\int_{\mathbb{T}^d}\widetilde{W} m\dy$, therefore
\[
(Id-\mathcal{B}^2_{\lambda})^{-1}(\widetilde{W})=\widetilde{W}+\frac{\beta}{1-\beta}\int_{\mathbb{T}^d}\widetilde{W} m\dy,
\]
and so
\[
\|(Id-\mathcal{B}^2_{\lambda})^{-1}(\widetilde{W})\|_{L^2}\leq\left(1+\left|\frac{\beta}{1-\beta}\right|\right)
\|m\|_{L^2}\|\widetilde{W}\|_{L^2},
\]
which means $Id-\mathcal{B}^2_{\lambda}\colon L^2(\mathbb{T}^d,\Rr^d)\to \colon L^2(\mathbb{T}^d,\Rr^d)$ is invertible.

Now we check the Assumption (B\ref{monotcond}).
For a point $I(x)=(x,P(x),m,V)$ with $P(x)=Du_{\lambda}, m=m_{\lambda},V=V_ {\lambda}$, where $(u_{\lambda}, m_{\lambda}, V_ {\lambda} )$ is a solution to \eqref{mainl},  we have $\mathcal{H}_{\lambda,I}\colon H^k(\mathbb{T}^d,\Rr^d)\times H^k(\mathbb{T}^d,\Rr)\times L^2(\mathbb{T}^d,\Rr^d)\to\Rr$ defined by
\begin{equation}
\label{hli}
\begin{split}
\mathcal{H}_{\lambda,I}(Q,f,W)=\int_{\mathbb{T}^d}[&m(x)Q(x)\mathcal{B}^0_{\lambda,I(x)}(Q(x))+
m(x)Q(x)\mathcal{B}^1_{\lambda,I(x)}(f)+g'(m(x))f^2(x)\\&-f(x)\mathcal{A}^1_{\lambda,I(x)}(f)+
m(x)Q(x)\mathcal{B}^2_{\lambda,I(x)}(W)-
f(x)\mathcal{A}^2_{\lambda,I(x)}(W)]\dx.
\end{split}
\end{equation}
Simple computations give
\[
\int_{\mathbb{T}^d}m(x)Q(x)\mathcal{B}^0_{\lambda,I(x)}(Q(x))\dx=
\int_{\mathbb{T}^d}m(x)Q(x)D^2_{pp}h_{\lambda}(P(x),x)Q(x)\dx\geq\sigma_1\int_{\mathbb{T}^d}m(x)|Q(x)|^2\dx,
\]
where $\sigma_1=\min\{1,\sigma\}$,
\[
\int_{\mathbb{T}^d}m(x)Q(x)\mathcal{B}^1_{\lambda,I(x)}(f)\dx=\int_{\mathbb{T}^d}[m(x)Q(x)\mathcal\beta\int_{\mathbb{T}^d}V(y)f(y)\dy ]\dx=
\beta\int_{\mathbb{T}^d}m(x)Q(x)\dx\int_{\mathbb{T}^d}Vf\dx,
\]

\[
\int_{\mathbb{T}^d}f\mathcal{A}^1_{\lambda,I(x)}(f)\dx=\int_{\mathbb{T}^d}\left[f(x)\beta P(x)\cdot\int_{\mathbb{T}^d}V(y)f(y)\dy\right]\dx=
\beta\int_{\mathbb{T}^d}P(x)f(x)\dx\int_{\mathbb{T}^d}V(x)f(x)\dx,
\]
\begin{align*}
\int_{\mathbb{T}^d}m(x)Q(x)\mathcal{B}^2_{\lambda,I(x)}(W)\dx
&=\int_{\mathbb{T}^d}\left[m(x)Q(x)\beta\int_{\mathbb{T}^d}W(y)m(y)\dy\right] \dx\\
&=\beta\int_{\mathbb{T}^d}m(x)Q(x)\dx\int_{\mathbb{T}^d}W(x)m(x)\dx,
\end{align*}

\[
\int_{\mathbb{T}^d}f(x)\mathcal{A}^2_{\lambda,I(x)}(W)\dx
=\int_{\mathbb{T}^d}\left[f(x)\beta P(x)\cdot\int_{\mathbb{T}^d}W(y)m(y)\dy\right] \dx
=\beta\int_{\mathbb{T}^d}P(x)f(x)\dx\int_{\mathbb{T}^d}W(x)m(x)\dx.
\]
Plugging all this into \eqref{hli} we obtain
\begin{equation*}
  \begin{split}
&\mathcal{H}_{\lambda,I}(Q,f,W)\geq\sigma_1\int_{\mathbb{T}^d}m|Q|^2\dx+\int_{\mathbb{T}^d}g'(m)|f|^2\dx
+\beta\int_{\mathbb{T}^d}mQ\dx\int_{\mathbb{T}^d}Vf\dx
\\&-\beta\int_{\mathbb{T}^d}Pf\dx\int_{\mathbb{T}^d}Vf\dx
+\beta\int_{\mathbb{T}^d}mQ\dx\int_{\mathbb{T}^d}Wm\dx
-\beta\int_{\mathbb{T}^d}Pf\dx\int_{\mathbb{T}^d}Wm\dx
  \end{split}
\end{equation*}
Since $(u_{\lambda}, m_{\lambda}, V_ {\lambda} )$ is a solution to \eqref{mainl},  we have that $\|m\|_{\infty}\|P\|_{\infty},\,\|V\|_{\infty}\leq \overline{C},\,m\geq \bar{m}$, for some constants $\overline{C},\bar{m}>0$. Since $g$ is strictly increasing we have that $g'(m)>\eta_0>0$ for all $m\in[\bar{m},\overline{C}]$. Then using Cauchy's and Holder's inequalities we get
\[
\left|
\int_{\mathbb{T}^d}mQ\dx\int_{\mathbb{T}^d}Vf\dx
\right|
\leq \frac{1}{2}\left[\left(\int_{\mathbb{T}^d}mQ\dx\right)^2+\left(\int_{\mathbb{T}^d}Vf\dx\right)^2\right]\leq
\frac{1}{2}\left(\int_{\mathbb{T}^d}m|Q|^2\dx+\overline{C}^2\int_{\mathbb{T}^d}|f|^2\dx\right),
\]
similarly
\[\left|
\int_{\mathbb{T}^d}Pf\dx\int_{\mathbb{T}^d}Vf\dx
\right|
\leq \overline{C}^2\int_{\mathbb{T}^d}|f|^2\dx,
\]

\[\left|
\int_{\mathbb{T}^d}mQ\dx\int_{\mathbb{T}^d}Wm\dx\right|\leq \frac{1}{2}\left[\int_{\mathbb{T}^d}m|Q|^2\dx+\left(\int_{\mathbb{T}^d}Wm\dx\right)^2\right],
\]

\[\left|
\int_{\mathbb{T}^d}Pf\dx\int_{\mathbb{T}^d}Wm\dx\right|\leq \frac{1}{2}\left[\overline{C}^2\int_{\mathbb{T}^d}|f|^2\dx+\left(\int_{\mathbb{T}^d}Wm\dx\right)^2\right].
\]
This yields
\[
\mathcal{H}_{\lambda,I}(Q,f,W)\geq\sigma_1\int_{\mathbb{T}^d}m|Q|^2\dx+\int_{\mathbb{T}^d}g'(m)|f|^2\dx-|\alpha| \int_{\mathbb{T}^d}m|Q|^2\dx-\frac{3}{2}|\alpha|\overline{C}^2\int_{\mathbb{T}^d}|f|^2\dx-|\alpha| \left(\int_{\mathbb{T}^d}Wm\dx\right)^2,
\]
thus
\[
\mathcal{H}_{\lambda,I}(Q,f,W)\geq(\sigma_1-|\alpha|)\int_{\mathbb{T}^d}m|Q|^2\dx+(\eta_0-\frac{3}{2}|\alpha|\overline{C}^2)\int_{\mathbb{T}^d}|f|^2\dx
-|\alpha| \left(\int_{\mathbb{T}^d}Wm\dx\right)^2.
\]
To estimate the last term let
\begin{equation}
R=W-\mathcal{B}^0_{\lambda,I(x)}(Q)-\mathcal{B}^1_{\lambda,I(x)}(f)-\mathcal{B}^2_{\lambda,I(x)}(W)=
W-D^2_{pp}h_{\lambda} Q-\beta\int_{\mathbb{T}^d}Vf\dx-\beta\int_{\mathbb{T}^d}Wm\dx.
\end{equation}
Integrating with respect to $m$ we get
\[
\int_{\mathbb{T}^d}Rm\dx=(1-\beta)\int_{\mathbb{T}^d}Wm\dx-\int_{\mathbb{T}^d}
D^2_{pp}h_{\lambda} Qm\dx-\beta\int_{\mathbb{T}^d}Vf\dx,
\]
\[
\int_{\mathbb{T}^d}Wm\dx=\frac{1}{1-\beta}\left[\int_{\mathbb{T}^d}Rm\dx+\int_{\mathbb{T}^d}
D^2_{pp}h_{\lambda} Qm\dx+\beta\int_{\mathbb{T}^d}Vf\right],
\]
hence  the inequality $(a+b+c)^2\leq 3 (a^2+b^2+c^2)$, Holder's inequality and the bound $|V|\leq \overline{C}$ yield
\[
\left(\int_{\mathbb{T}^d}Wm\dx\right)^2\leq\frac{3}{(1-\beta)^2}
\left[\int_{\mathbb{T}^d}R^2m\dx+\int_{\mathbb{T}^d}|D^2_{pp}h_{\lambda} Q|^2m\dx+|\beta|^2\overline{C}^2\int_{\mathbb{T}^d}|f|^2\right].
\]
If $|\alpha|<1/2$ then $|\beta|\leq|\alpha|<1$ and $|1-\beta|\geq 1-|\beta|\geq 1-|\alpha|<1/2$, thus
\[
\left(\int_{\mathbb{T}^d}Wm\dx\right)^2\leq 12
\left[\int_{\mathbb{T}^d}R^2m\dx+\int_{\mathbb{T}^d}|D^2_{pp}h_{\lambda} Q|^2m\dx+\overline{C}^2\int_{\mathbb{T}^d}|f|^2\right],
\]
and we get
\[
\mathcal{H}_{\lambda,I}(Q,f,W)\geq(\sigma_1-13|\alpha|)\int_{\mathbb{T}^d}m |D^2_{pp}h_{\lambda} Q|^2\dx+
(\eta_0-\frac{27}{2}|\alpha|\overline{C}^2)\int_{\mathbb{T}^d}|f|^2\dx
-12 \int_{\mathbb{T}^d}|R|^2m\dx.
\]
Let  $\alpha_0=\min\{\,\frac{\sigma_1}{26},\frac{\eta_0}{27} \,\}$, $\theta=\min\{\,\frac{\sigma_1}{2},\frac{\eta_0}{2} \,\}$,
then for $|\alpha|\leq \alpha_0$ we obtain
\begin{equation*}
\begin{split}
&\mathcal{H}_{\lambda,I}(Q,f,W)\geq\frac{\sigma_1}{2}\int_{\mathbb{T}^d}m| Q|^2\dx+
\frac{\eta_0}{2}\int_{\mathbb{T}^d}|f|^2\dx
-12 \int_{\mathbb{T}^d}|R|^2m\dx\\&
\geq\theta\int_{\mathbb{T}^d}m|Q|^2+|f|^2\dx-12 \int_{\mathbb{T}^d}|R|^2m\dx,
\end{split}
 \end{equation*}
 so Assumption (B\ref{monotcond}) holds.
\end{proof}

%\newpage

\bibliographystyle{alpha}

\bibliography{mfg}

\end{document}